\documentclass[12pt,a4paper]{amsart}
\usepackage[english]{babel}
\usepackage[T1]{fontenc}
\usepackage[latin1]{inputenc}
\usepackage{graphicx}
\usepackage{amsmath,amssymb,amsthm,mathrsfs,,txfonts}
\usepackage{soul}
\usepackage{array, multirow}
\usepackage{stmaryrd}
\usepackage{mathrsfs}
\usepackage{color}
\usepackage{hyperref}
\usepackage{enumerate}
\usepackage[all]{xy}
\usepackage[left=2cm, right=2cm, top=3cm, bottom=3cm]{geometry}

\theoremstyle{plain}
\newtheorem{theo}{Theorem}[section]
\newtheorem{lemme}[theo]{Lemma}
\newtheorem{prop}[theo]{Proposition}
\newtheorem{coro}[theo]{Corollary}

\theoremstyle{definition}

\newtheorem{nota}[theo]{Notation}

\newtheorem{conj}[theo]{Conjecture}
\newtheorem{exem}[theo]{Example}

\theoremstyle{remark}
\newtheorem{rema}[theo]{Remark}

\newtheorem{rappels}[theo]{Recall}

\def\A{\mathop{\rm A}\nolimits}

\def\E{\mathop{\rm E}\nolimits}
\def\F{\mathop{\rm F}\nolimits}
\def\G{\mathop{\rm G}\nolimits}
\def\H{\mathop{\rm H}\nolimits}

\def\K{\mathop{\rm K}\nolimits}
\def\L{\mathop{\rm L}\nolimits}
\def\M{\mathop{\rm M}\nolimits}

\def\O{\mathop{\rm O}\nolimits}

\def\S{\mathop{\rm S}\nolimits}
\def\T{\mathop{\rm T}\nolimits}
\def\U{\mathop{\rm U}\nolimits}

\def\Sp{\mathop{\rm Sp}\nolimits}
\def\det{\mathop{\rm det}\nolimits}
\def\GL{\mathop{\rm GL}\nolimits}	
\def\End{\mathop{\rm End}\nolimits}
\def\Hom{\mathop{\rm Hom}\nolimits}
\def\Op{\mathop{\rm Op}\nolimits}

\def\Re{\mathop{\rm Re}\nolimits}
\def\Im{\mathop{\rm Im}\nolimits}
\def\tr{\mathop{\rm tr}\nolimits}

\def\Ad{\mathop{\rm Ad}\nolimits}
\def\ad{\mathop{\rm ad}\nolimits}
\def\dim{\mathop{\rm dim}\nolimits}
\def\rk{\mathop{\rm rk}\nolimits}
\def\sgn{\mathop{\rm sgn}\nolimits}
\def\diag{\mathop{\rm diag}\nolimits}
\def\reg{\mathop{\rm reg}\nolimits}
\def\supp{\mathop{\rm supp}\nolimits}
\def\pr{\mathop{\rm pr}\nolimits}
\def\ch{\mathop{\rm ch}\nolimits}
\def\sh{\mathop{\rm sh}\nolimits}
\def\Chc{\mathop{\rm Chc}\nolimits}

\def\cos{\mathop{\rm cos}\nolimits}
\def\sin{\mathop{\rm sin}\nolimits}
\def\Id{\mathop{\rm Id}\nolimits}
\def\HS{\mathop{\rm HS}\nolimits}
\def\Mat{\mathop{\rm Mat}\nolimits}
\def\Cont{\mathop{\rm Cont}\nolimits}

\def\Id{\mathop{\rm Id}\nolimits}

\def\b{\mathop{\rm b}\nolimits}

\def\WF{\mathop{\rm WF}\nolimits}
\def\cov{\mathop{\rm cov}\nolimits}

\title{Transfer of characters in the theta correspondence with one compact member}

\author{ALLAN MERINO}
\address{ Department of Mathematics \\ National University of Singapore \\ Block S17 \\ 10, Lower Kent Ridge Road \\ Singapore 119076 \\ Republic of Singapore}
\email{matafm@nus.edu.sg}
\keywords{Howe correspondence, Characters, Oscillator semigroup, Reductive dual pairs}
\subjclass[2010]{Primary: 22E45; Secondary: 22E46, 22E30.}

\allowdisplaybreaks

\begin{document}

\maketitle

\begin{abstract}

For an irreducible dual pair $(\G, \G') \in \Sp(W)$ with one member compact and two representations $\Pi \leftrightarrow \Pi'$ appearing in the Howe duality, we give an expression of the character $\Theta_{\Pi'}$ of $\Pi'$ via the character of $\Pi$. We make computations for the dual pair $(\G = \U(n, \mathbb{C}), \G' = \U(p, q, \mathbb{C}))$, which are explicit in low dimensions. For $(\G = \U(1, \mathbb{C}), \G' = \U(1, 1, \mathbb{C}))$, we verify directly a result of H. Hecht saying that the character has the same value on both Cartan subgroups of $\G'$.

\end{abstract}

\tableofcontents

\section{Introduction}

\noindent For a finite dimensional representation $(\Pi, V)$ of a group G, the character of $\Pi$, denote by $\Theta_{\Pi}$, is defined by:
\begin{equation*}
\Theta_{\Pi}: \G \ni g \to \tr(\Pi(g)) \in \mathbb{C}.
\end{equation*}
In general, to determine precisely the character of such representations it's a hard problem, but in few cases, in particular for a compact connected group, the formula is explicit. Indeed, let G be a compact connected Lie group, T a Cartan subgroup of G, $\mathfrak{g}$ and $\mathfrak{t}$ the Lie algebras of G and T respectively, $\mathfrak{g}_{\mathbb{C}}$ and $\mathfrak{t}_{\mathbb{C}}$ the complexifications of $\mathfrak{g}$ and $\mathfrak{t}$, $\Phi(\mathfrak{g}_{\mathbb{C}}, \mathfrak{t}_{\mathbb{C}})$ (resp. $\Phi^{+}(\mathfrak{g}_{\mathbb{C}}, \mathfrak{t}_{\mathbb{C}})$) be the set of roots (resp. positive roots) of $\mathfrak{g}_{\mathbb{C}}$ with respect to $\mathfrak{t}_{\mathbb{C}}$ and $\mathscr{W} = \mathscr{W}(\mathfrak{g}_{\mathbb{C}}, \mathfrak{t}_{\mathbb{C}})$ be the corresponding Weyl group. According to H. Weyl, all the irreducible representations $(\Pi, V)$ of G are finite dimensional and parametrised by a linear form $\lambda$ on $\mathfrak{t}_{\mathbb{C}}$: this linear form is called the highest weight of $\Pi$. Moreover, the character of $\Pi$ is given by the following formula:
\begin{equation}
\Theta_{\Pi}(\exp(x)) = \sum\limits_{\omega \in \mathscr{W}} \sgn(\omega) \cfrac{e^{\omega(\lambda + \rho)(x)}}{\prod\limits_{\alpha \in \Phi^{+}(\mathfrak{g}_{\mathbb{C}}, \mathfrak{t}_{\mathbb{C}})} (e^{\frac{\alpha(x)}{2}} - e^{-\frac{\alpha(x)}{2}})} \qquad (x \in \mathfrak{t}^{\reg}),
\label{WeylCharacter}
\end{equation}
where $\rho$ is a linear form on $\mathfrak{t}_{\mathbb{C}}$ given by $\rho = \frac{1}{2} \sum\limits_{\alpha \in \Phi^{+}(\mathfrak{g}_{\mathbb{C}}, \mathfrak{t}_{\mathbb{C}})} \alpha$.

\noindent In the 50's, for a real reductive Lie group G, Harish-Chandra extended the concept of characters for a certain class of representation of G called quasi-simple (see \cite[Section~10]{HAR1}). More precisely, for such a representation $(\Pi, \mathscr{H})$ of G, he proved that the map:
\begin{equation*}
\Theta_{\Pi}: \mathscr{C}^{\infty}_{c}(\G) \ni \Psi \to \tr(\Pi(\Psi)) = \tr \displaystyle\int_{\G} \Psi(g) \Pi(g) dg \in \mathbb{C}
\end{equation*}
is well-defined and continuous. The map $\Theta_{\Pi}$ is called the global character of $\Pi$. Moreover, he proved that this distribution is given by an analytic function, still denoted by $\Theta_{\Pi}$, on the set of regular points of $\G$, i.e.
\begin{equation*}
\Theta_{\Pi}(\Psi) = \displaystyle\int_{\G} \Theta_{\Pi}(g) \Psi(g) dg \qquad (\Psi \in \mathscr{C}^{\infty}_{c}(\G)).
\end{equation*}
Again, an explicit formula for the function $\Theta_{\Pi}$ on $\G^{\reg}$ is hard to get. We recall briefly some well known facts on those characters. Let G be reductive group, K be a maximal compact subgroup of G such that $\rk(\K) = \rk(\G)$ and T be a Cartan subgroup of K (which is also a Cartan subgroup for G by our assumptions on ranks). As before, we denote by $\mathfrak{g}, \mathfrak{k}$ and $\mathfrak{t}$ their Lie algebras and by $\mathfrak{g}_{\mathbb{C}}, \mathfrak{k}_{\mathbb{C}}$ and $\mathfrak{t}_{\mathbb{C}}$ their complexifications.
\begin{enumerate}
\item If $(\Pi, \mathscr{H})$ is a discrete series representation of G of Harish-Chandra parameter $\lambda \in \mathfrak{t}^{*}_{\mathbb{C}}$, the character $\Theta_{\Pi}$ of $\Pi$ is given by (see \cite{HAR2}):
\begin{equation*}
\Theta_{\Pi}(\exp(x)) = (-1)^{\frac{\dim(\G/\K)}{2}} \sum\limits_{\omega \in \mathscr{W}(\mathfrak{k}_{\mathbb{C}}, \mathfrak{t}_{\mathbb{C}})} \varepsilon(\omega) \cfrac{e^{\omega(\lambda(x))}}{\prod\limits_{\alpha \in \Phi^{+}(\mathfrak{g}_{\mathbb{C}}, \mathfrak{t}_{\mathbb{C}})} (e^{\frac{\alpha(x)}{2}} - e^{-\frac{\alpha(x)}{2}})} \qquad (x \in \mathfrak{t}^{\reg}),
\end{equation*}
where $\mathscr{W}(\mathfrak{k}_{\mathbb{C}}, \mathfrak{t}_{\mathbb{C}})$ is the compact Weyl group of $\K$.
\item If $(\Pi, \mathscr{H}_{\Pi})$ is an irreducible unitary representation of $\G$ of highest weight $\lambda - \rho$, the character $\Theta_{\Pi}$ of $\Pi$ is given by (see \cite[Corollary~2.3]{ENR}):
\begin{equation*}
\prod\limits_{\alpha \in \Phi^{+}(\mathfrak{g}_{\mathbb{C}}, \mathfrak{t}_{\mathbb{C}}) \setminus \Phi^{+}(\mathfrak{k}_{\mathbb{C}}, \mathfrak{t}_{\mathbb{C}})}(e^{\frac{\alpha(x)}{2}} - e^{-\frac{\alpha(x)}{2}}) \Theta_{\Pi}(\exp(x)) = \sum\limits_{\omega \in \mathscr{W}^{\mathfrak{k}}_{\lambda}} (-1)^{l_{\lambda}(\omega)} \Theta(\K, \Lambda(\omega, \lambda))(\exp(x)) \qquad (x \in \mathfrak{t}^{\reg}),
\end{equation*}
where $\mathscr{W}^{\mathfrak{k}}_{\lambda}$ is defined in \cite[Definition~2.1]{ENR}, and $\Theta(\K, \Lambda(\omega, \lambda))(\exp(x))$ is the character of a $\K$-representation of highest weight $\Lambda(\omega, \lambda)$, where $\Lambda(\omega, \lambda)$ is defined in \cite[Corollary~2.3]{ENR}.
\end{enumerate}
We can also mention a conjecture of A. Kirillov (see \cite{KIR}), which should hold for a really general Lie group, and a paper of H. Hecht (\cite{HEC})

\noindent Let $(W, \langle \cdot, \cdot\rangle )$ be a real symplectic space, $\Sp(W)$ its corresponding group of isometries, $\widetilde{\Sp(W)}$ the metaplectic group and $(\omega, \mathscr{H})$ the corresponding metaplectic representation (see Section \ref{Metaplectic}). For a subgroup $\H \in \Sp(W)$, we denote by $\widetilde{\H}$ its preimage in $\widetilde{\Sp(W)}$ and by $\mathscr{R}(\widetilde{\H}, \omega)$ the set of equivalence classes of irreducible admissible representations of $\widetilde{\H}$ which are infinitesimally equivalent to a quotient of $\omega^{\infty}$. For an irreducible reductive dual pair $(\G, \G')$ in $\Sp(W)$, R. Howe proved that there exists a bijection between $\mathscr{R}(\widetilde{\G}, \omega)$ and $\mathscr{R}(\widetilde{\G'}, \omega)$.

\noindent In this paper, we assume that $\G$ is compact. In that case, it turns out that the situation is easier: the representations $\Pi$ and $\Pi'$ are just subrepresentations of $\omega^{\infty}$. Our goal here is to determine the character $\Theta_{\Pi'}$ using the character $\Theta_{\Pi}$ of $\Pi$. By projecting on the $\Pi$-isotypic component in $\mathscr{H}^{\infty}$, we get, for all $\Psi \in \mathscr{C}^{\infty}_{c}(\widetilde{\G})$, that:
\begin{equation*}
\Theta_{\Pi}(\Psi) = \tr \displaystyle\int_{\widetilde{\G'}} \displaystyle\int_{\widetilde{\G}} \overline{\Theta_{\Pi}(\widetilde{g})} \Psi(\widetilde{g'}) \omega^{\infty}(\widetilde{g}\widetilde{g'}) d\widetilde{g}d\widetilde{g'}.
\end{equation*}
So, formally, we have:
\begin{equation*}
\Theta_{\Pi'}(\widetilde{g}') = \displaystyle\int_{\widetilde{\G}} \overline{\Theta_{\Pi}(\widetilde{g})}\Theta(\widetilde{g}\widetilde{g}')d\widetilde{g} \qquad (\widetilde{g}' \in \widetilde{\G}'^{\reg}),
\end{equation*}
where the last equality is in terms of distributions on $\mathscr{C}^{\infty}_{c}(\widetilde{\G})$ and where $\Theta$ is the character of $\omega$ (see Section \ref{Metaplectic}). To avoid the problem of non-continuity of $\Theta$, we use the Oscillator semigroup introduced by Howe (see \cite{HOW3} or Section \ref{Oscillator}) and denoted by $\widetilde{\Sp(W_{\mathbb{C}})^{++}}$. The extension of $\Theta$ on $\widetilde{\Sp(W_{\mathbb{C}})^{++}}$ is holomorphic and $\Sp(W). \widetilde{\Sp(W_{\mathbb{C}})^{++}} \subseteq \widetilde{\Sp(W_{\mathbb{C}})^{++}}$. In particular, we get for $\widetilde{g'} \in \widetilde{\G'}^{\reg}$ that:
\begin{equation*}
\Theta_{\Pi'}(\widetilde{g'}) = \lim\limits_{\underset{\widetilde{p} \in \widetilde{\G'^{++}}}{\widetilde{p} \to 1}} \displaystyle\int_{\widetilde{\G}} \overline{\Theta_{\Pi}(\widetilde{g})} \Theta(\widetilde{g}\widetilde{g'}\widetilde{p}) d\widetilde{g},
\end{equation*}
where $\G'^{++} = \G'_{\mathbb{C}} \cap \Sp(W_{\mathbb{C}})^{++}$. The character of the representation $\Pi$ can be obtained using Weyl's character formula (see Equation \eqref{WeylCharacter}) together with a paper of M. Kashiwara and M. Vergne \cite{VER}, where they give explicitly the weights of the representations appearing in the correspondence. Moreover, using \cite{THO} or \cite{TOM1}, we get an explicit formula for the restriction of the character $\Theta$ on $\T. \T'^{++}$, where $\T$ (resp. $\T'$) is a compact torus of $\G$ (resp. $\G'$) and $\T'^{++} = \T'_{\mathbb{C}} \cap \Sp(W_{\mathbb{C}})^{++}$. We focus our attention on the case $(\G = \U(n, \mathbb{C}), \G' = \U(p, q, \mathbb{C}))$. More precisely, for $n=1$, we get the following result (see Proposition \ref{CharacterU(1,C)}):
\begin{equation*}
\Theta_{\Pi'_{k}}(\widetilde{t}') = \begin{cases} \prod\limits_{i=1}^{p+q} t^{\frac{1}{2}}_{i} \sum\limits_{h=1}^{p} \cfrac{t^{p-(k+1)}_{h}}{\prod\limits_{h \neq j} (t_{h} - t_{j})} \quad & \text { if } k \leqslant p - 1 \\
- \prod\limits_{i=1}^{p+q} t^{\frac{1}{2}}_{i} \sum\limits_{h = p+1}^{p+q} \cfrac{t^{p - (k+1)}_{h}}{\prod\limits_{h \neq j} (t_{h} - t_{j})} \quad & \text{ otherwise }
\end{cases}
\end{equation*}
In Section \ref{SectionU(1,1)}, we work with the pair $(\G = \U(1, \mathbb{C}), \G' = \U(1, 1, \mathbb{C}))$ and determine the value of the character $\Theta_{\Pi'}$ on the non-compact torus of $\U(1, 1, \mathbb{C})$ (unique up to conjugation, see Section \ref{SectionU(1,1)}). In particular, we verify a result of H. Hecht \cite{HEC}, saying that the value of the character does not depend of the Cartan subgroup

\noindent In Section \ref{ConjectureT}, we recall a conjecture of T. Przebinda (see \cite{TOM3}) concerning the transfer of characters for a general dual pair $(\G, \G')$ and present in few words an ongoing project linked with recents works of T. Przebinda \cite{TOM2}.

\bigskip

\noindent \textbf{Acknowledgements:} A part of this paper was done during my thesis at the University of Lorraine under the supervision of Angela Pasquale (University of Lorraine) and Tomasz Przebinda (University of Oklahoma). I would like to thank them for the ideas and time they shared with me. I finished this paper during my stay at the National University of Singapore, as a Research Fellow under the supervision of Hung Yean Loke. I am supported by the grant R-146-000-261-114 (Reductive dual pair correspondences and supercuspidal representations).

\section{Metaplectic representation and Howe's duality theorem}

\label{Metaplectic}

\noindent As far as I know, the first construction of this metaplectic representation used the so-called Stone-Von Neumann theorem. Brieffly, let $(W, \langle \cdot, \cdot\rangle )$ be a real symplectic space and $\H(W)$ the space $W \oplus \mathbb{R}$ with group multiplication:
\begin{equation*}
(w_{1}, \lambda_{1}).(w_{2}, \lambda_{2}) = (w_{1} + w_{2}, \lambda_{1} + \lambda_{2} + \frac{1}{2}\langle w_{1}, w_{2}\rangle ), \qquad (w_{1}, w_{2} \in W, \lambda_{1}, \lambda_{2} \in \mathbb{R}).
\end{equation*}
Clearly, $\mathscr{Z}(\H(W)) = \{(0, \lambda), \lambda \in \mathbb{R}\} \approx \mathbb{R}$. According to the Stone- Von Neumann theorem, for every character non trivial character $\Psi$ of $\mathscr{Z}(\H(W))$, there exists, up to equivalence, a unique irreducible unitary representation of $\H(W)$ with central character $\Psi$. The group of isometries of $(W, \langle \cdot, \cdot\rangle )$, denoted by $\Sp(W)$, acts naturally on $\H(W)$ by
\begin{equation*}
g. (w, \lambda) = (g(w), \lambda) \qquad (g \in \Sp(W), (w, \lambda) \in \H(W)).
\end{equation*}
By fixing an irreducible unitary representation $(\Pi_{\lambda}, \mathscr{H}_{\lambda})$ of $\H(W)$ with infinitesimal character $\Psi_{\lambda}, \lambda \in \mathbb{R}$, we get that the map: 
\begin{equation*}
\Pi_{\lambda, g}(h) = \Pi_{\lambda}(g^{-1}(h)) \qquad (g \in \Sp(W), h \in \H(W)),
\end{equation*}
is an irreducible unitary representation of $\H(W)$ with infinitesimal character $\Psi_{\lambda}$, and then, by application of the Stone- Von Neumann theorem, there exists an operator $\omega_{\lambda}(g)$ such that:
\begin{equation*}
\omega_{\lambda}(g) \Pi_{\lambda}(h) \omega_{\lambda}(g)^{-1} = \Pi_{\lambda}(g^{-1}(h)).
\end{equation*}
In particular, we get a projective representation of $\Sp(W)$. One can prove that we get a representation $(\omega, \mathscr{H})$ of a non-trivial double cover of $\Sp(W)$, that we will denote by $\widetilde{\Sp(W)}$ (see \cite{WEIL}).

\noindent In this section, we give an explicit realisation of the metaplectic representation (using a paper of A-M. Aubert and T. Przebinda \cite{TOM1}). In particular, we get a formula for the character of this representation (one can also check the paper of T. Thomas \cite{THO}).

\noindent Let $\chi$ be the character of $\mathbb{R}$ given by $\chi(r) = e^{2i\pi r}$. We denote by $\mathfrak{sp}(W)$ the Lie algebra of $\Sp(W)$, i.e.
\begin{equation*}
\mathfrak{sp}(W) = \left\{X \in \End(W), \langle X(w), w'\rangle  + \langle w, X(w')\rangle  = 0, \thinspace (\forall w, w' \in W)\right\}.
\end{equation*}
Let $J$ be an element of $\mathfrak{sp}(W)$ satisfying $J^{2} = - \Id$ and such that the symmetric bilinear form $(w, w')$ defined by $(w, w') = \langle J(w), w'\rangle $ is positive definite. For all $g \in \Sp(W)$, we denote by $J_{g}$ the automorphism of $W$ given by $J_{g} = J^{-1}(g-1)$. One can check easily that the adjoint $J^{*}_{g}$ of $J_{g}$ with respect to the form $\left(\cdot, \cdot\right)$ is given by $J^{*}_{g} = Jg^{-1}(1-g)$ and that the restriction of $J_{g}$ to $J_{g}(W)$ is well defined and invertible. The metaplectic group is defined as:
\begin{equation}
\widetilde{\Sp(W)} = \left\{\widetilde{g} = (g, \xi) \in \Sp(W) \times \mathbb{C}^{*}, \xi^{2} = i^{\dim_{\mathbb{R}}(g-1)W} \det(J_{g})^{-1}_{J_{g}(W)}\right\}.
\label{MetaplecticGroup}
\end{equation}
The covering map $\pi: \widetilde{\Sp(W)} \ni (g, \xi) \to g \in \Sp(W)$ is the first projection and the multiplication law is defined by:
\begin{equation*}
(g_{1}, \xi_{1}).(g_{2}, \xi_{2}) = (g_{1}g_{2}, \xi_{1}\xi_{2}C(g_{1}, g_{2})),
\end{equation*}
where the cocycle $C: \Sp(W) \times \Sp(W) \to \mathbb{C}$ is defined in \cite[Proposition~4.13]{TOM1}. Using \cite[Equation~(3)]{TOM2}, we get that the absolute value of $C$ satisfies, for every $g_{1}, g_{2} \in \Sp(W)$, the following equations:
\begin{equation*}
|C(g_{1}, g_{2})| = \sqrt{\left|\cfrac{\det(J_{g_{1}})_{J_{g_{1}}(W)}\det(J_{g_{2}})_{J_{g_{2}}(W)}}{\det(J_{g_{1}g_{2}})_{J_{g_{1}g_{2}}(W)}}\right|}, \qquad \qquad \cfrac{C(g_{1}, g_{2})}{|C(g_{1}, g_{2})|} = \chi\left(\frac{1}{8} \sgn(q_{g_{1}, g_{2}})\right),
\end{equation*}
where $\sgn(q_{g_{1}, g_{2}})$ is the signature of the form $q_{g_{1}, g_{2}}$ defined by:
\begin{equation*}
q_{g_{1}, g_{2}}(u, v) = \frac{1}{2} \left(\langle c(g_{1})u, v\rangle  + \langle c(g_{2})u, v\rangle \right) \qquad (u, v \in (g_{1}-1)W \cap (g_{2}-1)W).
\end{equation*}
To simplify the notations, for all $g \in \Sp(W)$, we denote by $\chi_{c(g)}$ the form on $(g-1)W$ given by $\chi_{c(g)}(u) = \chi\left(\frac{1}{8} \langle c(g)u, u\rangle \right)$.

\noindent We now construct the metaplectic representation. We denote by $\S(W)$ the Schwarz space corresponding to $W$ and by $t: \Sp(W) \to \S^{*}(W)$, $\Theta: \widetilde{\Sp(W)} \to \mathbb{C}^{*}$ and $T: \widetilde{\Sp(W)} \to \S^{*}(W)$ defined by
\begin{equation*}
t(g) = \chi_{c(g)} \mu_{(g-1)W} \qquad \Theta(\widetilde{g}) = \xi, \qquad T(\widetilde{g}) = \Theta(\widetilde{g})t(g), \qquad (\widetilde{g} = (g, \xi)),
\end{equation*}
where $\mu_{(g-1)W} \in \S^{*}(W)$ is the Lebesgue measure on the space $(g-1)W$ such that the volume with respect to $\left(\cdot, \cdot\right)$ of the corresponding unit cube is $1$.

\noindent We now fix a complete polarisation $W = X \oplus Y$, i.e. a direct sum of two maximal isotropic subspaces of $W$. The Weyl transform $\mathscr{K}: \S(W) \to \S(X \times X)$ given by:
\begin{equation*}
\mathscr{K}(\eta)(x, x') = \displaystyle\int_{Y} \eta(x-x'+y) \chi\left(\frac{1}{2} \langle y, x+x'\rangle \right) dy
\end{equation*}
is an isomorphism and the extension of $\mathscr{K}$ to the corresponding space of tempered distributions $\mathscr{K}: \S^{*}(W) \to \S^{*}(X \times X)$ is still an isomorphism. Similarly, the map $\Op: \S(X \times X) \to \Hom(\S(X), \S^{*}(X))$ given by:
\begin{equation*}
\Op(K)v(x) = \displaystyle\int_{X} K(x, x') v(x') dx'
\end{equation*}
extends to isomorphism $\Op: \S^{*}(X \times X) \to \Hom(\S(X), \S^{*}(X))$. According to \cite[Section~4.8]{TOM1}, for every $\Psi \in \L^{2}(W)$, $\Op \circ \mathscr{K}(\Psi)$ is an Hilbert-Schmidt operator on $\L^{2}(X)$ and the map:
\begin{equation*}
\Op \circ \mathscr{K}: \L^{2}(W) \to \HS(\L^{2}(X))
\end{equation*}
is an isometry. We denote by $\omega: \widetilde{\Sp(W)} \to \U(\L^{2}(X))$ defined by:
\begin{equation*}
\omega = \Op \circ \mathscr{K} \circ T
\end{equation*}
is a unitary representation of $\widetilde{\Sp(W)}$, called metaplectic representation. Moreover, the function $\Theta$ defined previously is the character of $(\omega, \L^{2}(X))$ and the space of smooth vectors is $\S(X)$, the Schwartz space of $X$. 

\begin{rema}

We denote by $\Sp(W)^{c}$ the subspace of $\Sp(W)$ defined by $\Sp(W)^{c} = \{g \in \Sp(W), \det(g-1) \neq 0\}$: it's the domain of the Cayley transform of $\Sp(W)$. We denote by $\widetilde{\Sp(W)^{c}}$ the preimage of $\Sp(W)^{c}$ in $\widetilde{\Sp(W)}$. For every $\widetilde{g} = (g, \xi) \in \widetilde{\Sp(W)^{c}}$, we get:
\begin{equation*}
\Theta(\widetilde{g}) = \left(i^{\dim_{\mathbb{R}}W} \det(J(g-1))^{-1}\right)^{\frac{1}{2}} =  \det(i(g-1))^{-\frac{1}{2}}.
\end{equation*}

\end{rema}

\noindent A dual pair in $\Sp(W)$ is a pair of subgroups $(\G, \G')$ of $\Sp(W)$ which are mutually centralizer in $\Sp(W)$, i.e. $C_{\Sp(W)}(\G) = \G'$ and $C_{\Sp(W)}(\G') = \G$. The dual pair is said to be reductive if the action of $\G$ and $\G'$ on $W$ is reductive. If we have a decomposition of $W$ as an orthogonal sum $W = W_{1} \oplus W_{2}$ where $W_{1}$ and $W_{2}$ are $\G.\G'$-invariants, then $(\G_{|_{W_{i}}}, \G'_{|_{W_{i}}})$ is a dual pair in $\Sp(W_{|_{W_{i}}}), i= 1, 2$. If we cannot find such a decomposition, the dual pair is said irreducible.

\noindent The irreducible reductive dual pairs in the symplectic group had been classified by R. Howe \cite{HOW5}. In this paper, we assume that the group $\G$ is compact. In this case, $(\G, \G')$ is one of the following dual pairs
\begin{enumerate}
\item $(\U(n, \mathbb{C}), \U(p, q, \mathbb{C})) \subseteq \Sp(2n(p+q), \mathbb{R})$,
\item $(\O(n, \mathbb{R}), \Sp(2m, \mathbb{R})) \subseteq \Sp(2nm, \mathbb{R})$,
\item $(\U(n, \mathbb{H}), \O^{*}(m, \mathbb{H})) \subseteq \Sp(4nm, \mathbb{R})$.
\end{enumerate}
For the computations in the Section \ref{ComputationsUnitary}, we will focus our attention on the first one. 

\begin{rema}
\begin{enumerate}
\item For a dual pair $(\G, \G')$ in $\Sp(W)$, we denote by $\widetilde{\G} = \pi^{-1}(\G)$ and $\widetilde{\G}'$ the preimages of $\G$ and $\G'$ in $\widetilde{\Sp(W)}$. In \cite{HOW5}, R. Howe proved that $(\widetilde{\G}, \widetilde{\G}')$ is a dual pair in $\widetilde{\Sp(W)}$. With the precise definition we gave for $\widetilde{\Sp(W)}$ in Equation \eqref{MetaplecticGroup}, we can see that easily. Indeed, we need to prove that for all $g \in \G$ and $g' \in \G'$, we have $C(g, g') = C(g', g)$. Obviously, $|C(g, g')| = |C(g', g)|$, and because $q_{g, g'} = q_{g', g}$, the result follows.
\item If the group $\G$ is compact, then $\widetilde{\G}$ is also compact.
\end{enumerate}
\end{rema}

\noindent From now on, we assume that $(\G, \G')$ is an irreducible reductive dual pair in $\Sp(W)$ with G compact. The Howe duality theorem can be stated in a easier way when we assume that one member is compact. As before, we consider a complete polarisation of $W$ of the form $X \oplus Y$ and we realise the metaplectic representation $\omega$ on the space $\L^{2}(X)$. The space of smooth vector is the Schwartz space $\S(X)$ and under the action of $\widetilde{\G}$, we get the following decomposition:
\begin{equation*}
\S(X) = \bigoplus\limits_{(\Pi, V_{\Pi}) \in \widehat{\widetilde{\G}}_{\omega}} V(\Pi), 
\end{equation*}
where $\widehat{\widetilde{\G}}_{\omega}$ is the set of irreducible unitary representations of $\widetilde{\G}$ such that $\Hom_{\widetilde{\G}}(\Pi, \omega^{\infty}) \neq \{0\}$ and $V(\Pi)$ is the $\Pi$-isotypic component in $\S(X)$, i.e. the closure with respect to the topology on $\S(X)$ of the sspace $\{T(V_{\Pi}), T \in \Hom_{\widetilde{\G}}(\Pi, \omega^{\infty})\}$.

\noindent Because $\widetilde{\G}'$ commute with $\widetilde{\G}$, the group $\widetilde{\G}'$ acts on $V(\Pi)$ for every $\Pi \in \widehat{\widetilde{\G}}_{\omega}$, and as a $\widetilde{\G} \times \widetilde{\G}'$-module, we get the following decomposition:
\begin{equation*}
\S(X) = \bigoplus\limits_{(\Pi, V_{\Pi}) \in \widehat{\widetilde{\G}}_{\omega}} \Pi \otimes \Pi',
\end{equation*}
where $\Pi'$ is an irreducible unitary representation of $\widetilde{\G}'$. The map
\begin{equation*}
\theta: \widehat{\widetilde{\G}}_{\omega} \ni \Pi \to \Pi' = \theta(\Pi) \to \widehat{\widetilde{\G'}}_{\omega}
\end{equation*}
is one-to-one and usually called Howe's correspondence.

\noindent Let $(\Pi, V_{\Pi}) \in \widehat{\widetilde{\G}}_{\omega}$ and $\Pi' = \theta(\Pi)$ the corresponding representation of $\widetilde{\G}'$. We denote by $\mathscr{P}_{\Pi}: \S(X) \to V(\Pi)$ the projection onto the $\Pi$-isotypic component. According to \cite[Section~1.4]{WAL}, the map $\mathscr{P}_{\Pi}$ is given by the formula:
\begin{equation*}
\mathscr{P}_{\Pi} = d_{\Pi} \displaystyle\int_{\widetilde{\G}} \overline{\Theta_{\Pi}(\widetilde{g})} \omega^{\infty}(\widetilde{g}) d\widetilde{g} = \omega^{\infty}(d_{\Pi} \overline{\Theta_{\Pi}}),
\end{equation*}
where $d_{\Pi} = \dim_{\mathbb{C}}(V_{\Pi})$ is the dimension of the representation $\Pi$. We get the following result for the global character of $\Pi'$.

\begin{prop}

For every compactly supported function $\Psi \in \mathscr{C}^{\infty}_{c}(\widetilde{\G}')$, we get:
\begin{equation*}
\Theta_{\Pi'}(\Psi) = \tr \displaystyle\int_{\widetilde{\G}'}\left(\displaystyle\int_{\widetilde{\G}} \overline{\Theta_{\Pi}(\widetilde{g})}\omega^{\infty}(\widetilde{g}\widetilde{g}') d\widetilde{g}\right) \Psi(\widetilde{g}')d\widetilde{g}'.
\end{equation*}

\end{prop}

\begin{proof}

For such a function $\Psi$, we have:
\begin{equation*}
\tr(\mathscr{P}_{\Pi} \omega(\Psi)) = \tr(\Id_{V_{\Pi}} \otimes \Pi'(\Psi)) = d_{\Pi} \Theta_{\Pi'}(\Psi),
\end{equation*}
and then,
\begin{equation*}
\Theta_{\Pi'}(\Psi) = \frac{1}{d_{\Pi}} \tr(\mathscr{P}_{\Pi} \omega^{\infty}(\Psi)) = \tr \displaystyle\int_{\widetilde{\G}'} \Psi(\widetilde{g}') \mathscr{P}_{\Pi} \omega^{\infty}(\widetilde{g}') d\widetilde{g}' = \tr \displaystyle\int_{\widetilde{\G}'} \left(\displaystyle\int_{\widetilde{\G}}\overline{\Theta_{\Pi}(\widetilde{g})} \omega^{\infty}(\widetilde{g}\widetilde{g}') d\widetilde{g}\right) \Psi(\widetilde{g}')d\widetilde{g}'.
\end{equation*}

\end{proof}

\noindent Using that $\Theta$ is the character of $\omega$, we get formally:
\begin{equation*}
\Theta_{\Pi'}(\widetilde{g}) = \displaystyle\int_{\widetilde{\G}}\overline{\Theta_{\Pi}(\widetilde{g})} \Theta(\widetilde{g}\widetilde{g}') d\widetilde{g} \qquad (\widetilde{g}' \in \widetilde{\G}').
\end{equation*}
Because the character $\Theta$ is not continuous, the second member of the previous equation could not make sense. To avoid this problem, we use the Oscillator semigroup introduced by R. Howe (see \cite{HOW3}).

\section{Howe's Oscillator semigroup}

\label{Oscillator}

\noindent Let $(W, \langle \cdot, \cdot\rangle)$ be a (finite dimensional) real symplectic vector space and $(W_{\mathbb{C}}, \langle\cdot, \cdot\rangle)$ its complexification. For every $w \in W_{\mathbb{C}}$, we consider the decomposition $w = a + ib, a, b \in W$ and we denote by $\overline{w} = a - ib$ the conjugate with respect to the decomposition $W_{\mathbb{C}} = W \oplus iW$. By extension, we get a symplectic form $\langle \cdot, \cdot \rangle$ on $W_{\mathbb{C}}$.

\begin{lemme}

\noindent The form $\H: W_{\mathbb{C}} \times W_{\mathbb{C}} \to \mathbb{C}$ defined by
\begin{equation}
\H(w, w^{'}) = i\langle w, \overline{w}'\rangle
\label{HermitianForm}
\end{equation}
is hermitian.

\end{lemme}

\begin{proof}

Straightforward verification.

\end{proof}

\noindent We define now the subset $\Sp(W_{\mathbb{C}})^{++}$ of $\Sp(W_{\mathbb{C}})$ by:
\begin{equation}
\Sp(W_{\mathbb{C}})^{++} = \left\{g \in \Sp(W_{\mathbb{C}}) \thinspace ; \thinspace \H(w, w) > \H(g(w), g(w)), (\forall w \in W_{\mathbb{C}} \setminus \{0\})\right\}.
\label{SpWC++}
\end{equation}

\noindent Similarly, we denote by $\mathfrak{sp}(W_{\mathbb{C}})^{++}$ the subset of $\End(W_{\mathbb{C}})$ given by:
\begin{equation*}
\mathfrak{sp}(W_{\mathbb{C}})^{++} = \left\{z = x + iy \thinspace ; \thinspace x, y \in \mathfrak{sp}(W), \det(z-1) \neq 0, \langle yw, w\rangle > 0, w \in W \setminus \{0\}\right\}.
\end{equation*}

\begin{lemme}

Fix an element $z = x + iy$ with $x, y \in \mathfrak{sp}(W)$ such that $\det(z-1) \neq 0$. Then, 
\begin{equation*}
\H(w, w) > \H(c(z)w, c(z)w)  \qquad (\forall w \in W_{\mathbb{C}} \setminus \{0\})
\end{equation*}
if and only if
\begin{equation*}
\langle yw, w\rangle >0 \qquad (\forall w \in W \setminus \{0\}).
\end{equation*}
\label{Semi-groupe-Cayley}
Finally, we obtain $c(\mathfrak{sp}(W_{\mathbb{C}})^{++}) = \Sp(W_{\mathbb{C}})^{++}$. 
\end{lemme}

\begin{proof}

Fix $z = x + iy$, with $x, y \in \mathfrak{sp}(W)$. We have $\H(c(z)w, c(z)w) = \H(\overline{c(z)}^{-1}c(z)w, w)$ and then $\H(w, w) > \H(c(z)w, c(z)w) \Leftrightarrow \H((1 - \overline{c(z)}^{-1}c(z))w, w) > 0$. Or,
\begin{eqnarray*}
1 - \overline{c(z)}^{\,-1}c(z)) & = & 1 - \left((\bar{z} + 1)(\bar{z}-1)^{-1}\right)^{-1}(z+1)(z-1)^{-1} = 1 - (\bar{z}-1)(\bar{z}+1)^{-1}(z+1)(z-1)^{-1} \\
                                           & = & 1 - (\bar{z}+1)^{-1}(\bar{z}-1)(z+1)(z-1)^{-1} \\
                                           & = & (\bar{z}+1)^{-1}(\bar{z}+1)(z-1)(z-1)^{-1} - (\bar{z}+1)^{-1}(\bar{z}-1)(z+1)(z-1)^{-1} \\
                                           & = & (\bar{z}+1)^{-1}\left( (\bar{z}+1)(z-1) - (\bar{z}-1)(z+1)\right)(z-1)^{-1}
\end{eqnarray*}                                           
By definition of $z$, we get $(\bar{z}+1)(z-1) - (\bar{z}-1)(z+1) = 4iy$. So, $1 - \overline{c(z)}^{-1}c(z) = 4i(\bar{z}+1)^{-1}y(z-1)^{-1}$.
Then, for all $w \in W_{\mathbb{C}} \setminus \{0\}$, we get:
\begin{eqnarray*}
\H(w, w) > \H(c(z)w, c(z)w) & \Leftrightarrow & \H((1 - \overline{c(z)}^{-1}c(z))w, w) > 0  \Leftrightarrow  4i\H((\bar{z}+1)^{-1}y(z-1)^{-1}w, w) > 0 \\
                                         & \Leftrightarrow & -4\langle(\bar{z}+1)^{-1}y(z-1)^{-1}w, \overline{w}\rangle > 0  \Leftrightarrow  -4 \langle y(z-1)^{-1}w, (-\bar{z}+1)^{-1}\overline{w}\rangle > 0 \\
                                         & \Leftrightarrow & 4\langle yw^{'}, \overline{w}'\rangle > 0
\end{eqnarray*}
with $w^{'} = (z-1)^{-1}w$. Because $\langle yw^{'}, \overline{w}'\rangle \in \mathbb{R}^{*}_{+}$, by writing $w'$ as $w^{'} = w^{'}_{1} + iw^{'}_{2}$, we get: 
\begin{equation*}
\langle yw^{'}, \overline{w}'\rangle = \langle y(w^{'}_{1}), w^{'}_{1}\rangle + \langle y(w^{'}_{2}), w^{'}_{2}\rangle.
\end{equation*}

\end{proof}

\begin{prop}

The set $\Sp(W_{\mathbb{C}})^{++}$ is a subsemigroup of $\Sp(W_{\mathbb{C}})$, which does not contain the identity but stable under $g \to \bar{g}^{-1}$. Moreover, we have
\begin{equation}
\Sp(W_{\mathbb{C}})^{++}. \Sp(W) = \Sp(W). \Sp(W_{\mathbb{C}})^{++} \subseteq \Sp(W_{\mathbb{C}})^{++}
\label{Semi-groupe1}
\end{equation}
and the set $\Sp(W_{\mathbb{C}})^{++} \cup \Sp(W)$ is a subsemigroup of $\Sp(W_{\mathbb{C}})$. To conclude, the symplectic group $\Sp(W)$ is contained in the closure of $\Sp(W_{\mathbb{C}})^{++}$.

\end{prop}

\begin{proof}

Fix $g$ and $g'$ in $\Sp(W_{\mathbb{C}})^{++}$. Obviously, $gg' \in \Sp(W_{\mathbb{C}})$. For every $w \in W_{\mathbb{C}}$, we have:
\begin{equation*}
\H(gg'w, gg'w) < \H(g'w, g'w) < H(w, w),
\end{equation*}
which imply that $gg' \in \Sp(W_{\mathbb{C}})^{++}$. The subspace $\mathfrak{sp}(W_{\mathbb{C}})^{++}$ is stable under the map $z \to -\bar{z}$ and $\overline{c(z)}^{\,-1} = c(-\bar{z})$. Then, if $g \in \Sp(W_{\mathbb{C}})^{++}$, we get $\bar{g}^{-1} \in \Sp(W_{\mathbb{C}})^{++}$.

\noindent Now, fix $g \in \Sp(W_{\mathbb{C}})^{++}$ and $h \in \Sp(W)$. For all $w \in W_{\mathbb{C}}$, we have $\overline{h(w)} = h(\bar{w})$ and then:
\begin{equation*}
H(gh(w), \overline{gh(w)}) < H(h(w), \overline{h(w)}) = i\langle h(w), \overline{h(w)}\rangle = i\langle h(w), h(\bar{w})\rangle = H(w, w).
\end{equation*}
In particular, $gh \in \Sp(W_{\mathbb{C}})^{++}$. Finally, for every element $g \in \Sp(W)$,
\begin{equation*}
g = -c(0)g = \lim_{\underset{\langle y \cdot, \cdot \rangle > 0}{y \to 0}} -c(iy)g
\end{equation*}
which prove that every elements of $\Sp(W)$ is a limit of elements in the semigroup $\Sp(W_{\mathbb{C}})^{++}$.

\end{proof}

\begin{rema}

Let $z = X+iY$ with $z \in \mathfrak{sp}(W_{\mathbb{C}})^{++}$. Then, for all $w \in W$,
\begin{equation*}
\chi_{z}(w) = e^{\frac{i\pi}{2}w^{t}J(X+iY)w} = e^{\frac{i\pi}{2}w^{t}JXw} e^{-\frac{\pi}{2}w^{t}JYw}.
\end{equation*}
The matrix $Y \in \mathfrak{sp}(W)$, the form $\langle Y \cdot, \cdot\rangle$ is positive and $JY$ is symmetric and positive definite. Then, there exists a diagonal matrix $D = \diag(d_{1}, \ldots, d_{2n})$ and a matrix $O \in \O(2n, \mathbb{R})$ such that $JY = O^{t}DO$. So, 
\begin{eqnarray*}
\displaystyle\int_{W} \chi_{iY}(w) dw & = & \displaystyle\int_{W} e^{-\frac{\pi}{2}w^{t}JYw} dw = \displaystyle\int_{W} e^{-\frac{\pi}{2}w^{t}O^{t}DOw} dw = \displaystyle\int_{W} e^{-\frac{\pi}{2}Y^{t}DY} |\det(O)| dY \\
                                                          & = & \displaystyle\int_{W} e^{-\frac{\pi}{2}\sum\limits_{k=1}^{2n} d_{k}Y^{2}_{k}} |\det(O)| dY = \prod\limits_{k=1}^{2n} \displaystyle\int_{\mathbb{R}} e^{-\frac{\pi}{2}d_{k}Y^{2}_{k}} dY_{k} = \prod\limits_{k=1}^{2n} \cfrac{1}{\sqrt{d_{k}}} = \det^{-\frac{1}{2}}(D)
\end{eqnarray*}
Using that $|\chi_{X+iY}(w)| = e^{-\frac{\pi}{2}w^{t}JYw}$, we get that the integral 
\begin{equation*}
\displaystyle\int_{W} \chi_{X+iY}(w) dw
\end{equation*}
is absolutely convergent. More precisely, we get:
\begin{equation*}
\displaystyle\int_{W} \chi_{X+iY}(w) dw = \det^{-\frac{1}{2}}\left(\frac{1}{2}(X+iY)\right).
\end{equation*}     
From now on, we denote by $\Lambda(X+iY)$ the previous determinant, i.e.
\begin{equation}
\Lambda(X+iY) = \det^{-\frac{1}{2}}\left(\frac{1}{2}(X+iY)\right).
\label{EquationLambda}
\end{equation}                                           

\end{rema}

\noindent Even if the complex symplectic group $\Sp(W_{\mathbb{C}})$ is simply connected, the complex manifold $\Sp(W_{\mathbb{C}})^{++}$ is not simply connected. We define on $\Sp(W_{\mathbb{C}})^{++}$ a non-trivial cover, denoted by $\widetilde{\Sp(W_{\mathbb{C}})}^{++}$, by
\begin{equation*}
\widetilde{\Sp(W_{\mathbb{C}})}^{++} = \left\{(g, \xi) \thinspace ; \thinspace g \in \Sp(W_{\mathbb{C}})^{++}, \xi^{2} = \det(i(g-1))^{-1}\right\},
\end{equation*}
and let $C: \Sp(W_{\mathbb{C}})^{++} \times \Sp(W_{\mathbb{C}})^{++} \to \mathbb{C}$ defined by:
\begin{equation*}
C(g_{1}, g_{2}) = \det^{-\frac{1}{2}}\left(\frac{1}{2i}(c(g_{1}) + c(g_{2}))\right).
\end{equation*}

\begin{theo}

The function $\Theta : \widetilde{\Sp(W_{\mathbb{C}})}^{++} \ni (g, \xi) \to \xi \in \mathbb{C}$ is holomorphic, and we have the following equality:
\begin{equation}
\cfrac{\Theta(\widetilde{g_{1}}\widetilde{g_{2}})}{\Theta(\widetilde{g_{1}})\Theta(\widetilde{g_{2}})} = C(g_{1}, g_{2}) \qquad \left(\widetilde{g_{1}}, \widetilde{g_{2}} \in \widetilde{\Sp(W_{\mathbb{C}})}^{++} \cup \widetilde{\Sp(W)}\right).
\label{Equation Theta}
\end{equation}

\noindent Moreover, for every functions $\Psi \in \mathscr{C}^{\infty}_{c}(\widetilde{\Sp(W)})$, we get:
\begin{equation}
\displaystyle\int_{\widetilde{\Sp(W)}} \Theta(\widetilde{g})\Psi(\widetilde{g}) d\mu_{\widetilde{\Sp(W)}}(\widetilde{g}) = \lim\limits_{\underset{\widetilde{p} \in \widetilde{\Sp(W_{\mathbb{C}})}^{++}}{\widetilde{p} \to 1}} \displaystyle\int_{\widetilde{\Sp(W)}} \Theta(\widetilde{p}\widetilde{g}) \Psi(\widetilde{g}) d\mu_{\widetilde{\Sp(W)}}(\widetilde{g}).
\label{Limite Caractere}
\end{equation}
\label{Theoreme section 3 important}
\end{theo}

\begin{proof}

We assume that the support of $\Psi$ is contained in the image of $\widetilde{c}(0) \widetilde{c}$. Then,
\begin{eqnarray*}
\displaystyle\int_{\widetilde{\Sp(W)}} \Theta(\widetilde{p}\widetilde{g})\Psi(\widetilde{g}) d\widetilde{g} & = & \displaystyle\int_{\widetilde{\Sp(W)}} \Theta(\widetilde{p}\widetilde{c}(0)\widetilde{g})\Psi(\widetilde{c}(0)\widetilde{g}) d\widetilde{g} = \displaystyle\int_{\supp \Psi} \Theta(\widetilde{p}\widetilde{c}(0)\widetilde{g})\Psi(\widetilde{c}(0)\widetilde{g}) d\widetilde{g} \\
                               & = & \displaystyle\int_{\mathfrak{sp}(W)} \Theta(\widetilde{p}\widetilde{c}(0)\widetilde{c}(x))\Psi(\widetilde{c}(0)\widetilde{c}(x)) j(x) dx = \displaystyle\int_{\mathfrak{sp}(W)} \Theta(\widetilde{c}(iy)\widetilde{c}(x))\Psi(\widetilde{c}(0)\widetilde{c}(x)) j(x) dx
\end{eqnarray*}
where $\widetilde{p}\widetilde{c}(0) = \widetilde{c}(iy)$ with $y \in \mathfrak{sp}(W)$ and $\langle y\cdot, \cdot\rangle > 0$. In particular, $y \to 0$ when $\widetilde{p} \to 1$. But, according to equation \eqref{EquationLambda}
\begin{equation*}
\Theta(\widetilde{c}(iy) \widetilde{c}(x)) = \Theta(\widetilde{c}(iy)) \Theta(\widetilde{c}(x)) \Lambda(x+iy).
\end{equation*}  
We denote by $\psi$ the function of $\mathfrak{sp}(W)$ given by $\psi(x) = \Theta(\widetilde{c}(x)) \Psi(\widetilde{c}(x)) j(x)$ (we notice easily that $\psi \in \mathscr{C}^{\infty}_{c}(\mathfrak{sp}(W))$). We get:
\begin{eqnarray*}
\displaystyle\int_{\mathfrak{sp}(W)} \Lambda(x+iy) \psi(x)dx & = & \displaystyle\int_{\mathfrak{sp}(W)} \displaystyle\int_{W} \chi_{x +iy}(w) \psi(x) dwdx = \displaystyle\int_{W} \displaystyle\int_{\mathfrak{sp}(W)} \chi_{x}(w) \chi_{iy}(w) \psi(x) dxdw \\
                                  & = & \displaystyle\int_{W} \chi_{iy}(w) \displaystyle\int_{\mathfrak{sp}(W)} \chi_{x}(w) \psi(x) dxdw
\end{eqnarray*}
and then
\begin{equation*}
\displaystyle\int_{\mathfrak{sp}(W)} \chi_{x}(w) \psi(x)dx = \displaystyle\int_{\mathfrak{sp}(W)} \psi(x) e^{2i\pi \tau(w)(x)} dx = \widehat{\psi}\left(\frac{1}{4} \tau(w)\right),
\end{equation*}
where $\tau : W \to \mathfrak{sp}(W)^{*}$ is the moment map and $\widehat{\psi}$ is the Fourier transform of $\psi$ on $\mathfrak{sp}(W)$. Then,
\begin{equation*}
\displaystyle\int_{\mathfrak{sp}(W)} \Lambda(x+iy) \psi(x)dx = \displaystyle\int_{W} \chi_{iy}(w) \widehat{\psi}\left(\frac{1}{4} \tau(w)\right) dw.
\end{equation*}
For all $w \in W \setminus \{0\}$, we have
\begin{equation*}
\chi_{iy}(w) = e^{\frac{2i\pi}{4}\langle iy(w), w\rangle} = e^{-\frac{\pi}{2}\langle yw, w\rangle} < 1,
\end{equation*}
because $\langle yw, w\rangle > 0$ for every non zero $w \in W$. Finally, we get:
\begin{equation*}
\lim\limits_{y \to 0} \displaystyle\int_{\mathfrak{sp}(W)} \Lambda(x+iy) \psi(x)dx = \displaystyle\int_{W} \widehat{\psi}\left(\frac{1}{4} \tau(w)\right) dw.
\end{equation*}
Using that $\lim\limits_{y \to 0} \Theta(\widetilde{c}(iy)) = \Theta(\widetilde{c}(0))$, we get:
\begin{equation*}
\lim\limits_{y \to 0} \displaystyle\int_{\mathfrak{sp}(W)} \Theta(\widetilde{c}(iy)) \Lambda(x+iy) \psi(x)dx = \displaystyle\int_{W} \widehat{\psi}\left(\frac{1}{4} \tau(w)\right) dw.
\end{equation*}
Finally, we proved that the limit we considered in Equation \eqref{Limite Caractere} exists. Now, we determine this limit. For every $x \in \mathfrak{sp}(W)$, we denote by $B$ the matrix of the bilinear form $\langle x \cdot, \cdot \rangle$. We remark that the matrix of the form $\langle J \cdot, \cdot\rangle$ is the identity matrix. For all $t > 0$, we have:
\begin{eqnarray*}
\Lambda(x+itJ) & = & \displaystyle\int_{W} \chi_{x+itJ}(w) dw = \displaystyle\int_{W} e^{\frac{i\pi}{2}\langle(x+itJ)w, w\rangle} dw = \displaystyle\int_{W} e^{\frac{i\pi}{2}\langle(x+itJ)w, w\rangle} dw \\
                         & = & \displaystyle\int_{W} e^{\frac{i\pi}{2}w^{t}(B+itI)w} dw = \displaystyle\int_{W} e^{\frac{-\pi}{2}w^{t}(-iB+tI)w} dw = \det^{-\frac{1}{2}}\left(\frac{1}{2}(-iB + tI)\right)
\end{eqnarray*}
We know that the eigenvalues of $x$ are real numbers. So, for all $t >0$, 
\begin{equation*}
|\det(-iB + tI)| > |\det(iB)|
\end{equation*}
i.e.
\begin{equation*}
|\det(-iB + tI)|^{-\frac{1}{2}} < |\det(iB)|^{-\frac{1}{2}}.
\end{equation*}
Then,
\begin{equation*}
|\Lambda(x+itJ)| \leq |\Lambda(x)|.
\end{equation*}
Using that the function $\Lambda$ is locally integrable, we get:
\begin{eqnarray*}
\lim\limits_{t \to 0^{+}} \displaystyle\int_{\mathfrak{sp}(W)}\Lambda(x+itJ) \psi(x) dx = \displaystyle\int_{\mathfrak{sp}(W)}\Lambda(x) \psi(x) dx.
\end{eqnarray*}

\end{proof}

\begin{rema}

We first extend the map $T$ on the semigroup. For every element $g \in \Sp(W_{\mathbb{C}})^{++}$, we have $\det(g-1)$. We define the map $T: \widetilde{\Sp(W_{\mathbb{C}})^{++}} \to \S^{*}(W)$ by
\begin{equation*}
T(\widetilde{g} = (g, \xi)) = \Theta(\widetilde{g}) \chi_{c(g)} \mu_{W}.
\end{equation*}
We denote by $\Cont(\L^{2}(X))$ the semigroup of contractions on the Hilbert space $\L^{2}(X)$. One can prove that:
\begin{equation*}
\omega = \Op \circ \mathscr{K} \circ T: \widetilde{\Sp(W_{\mathbb{C}})^{++}} \to \Cont(\L^{2}(X))
\end{equation*}
is a semigroup homomorphism. Moreover, for all $\widetilde{p} \in \widetilde{\Sp(W_{\mathbb{C}})^{++}}$, the operator $\omega(\widetilde{p})$ is of trace class and
\begin{equation*}
\tr \omega(\widetilde{p}) = \Theta(\widetilde{p}).
\end{equation*}

\end{rema}

\section{A general formula for $\Theta_{\Pi'}$}

\label{IntegralFormula}

\noindent Let us start this section with comments concerning some particular integrals.  As shown in \cite[Section~4.8]{TOM1}, for all $\Psi \in \mathscr{C}^{\infty}_{c}(\widetilde{\Sp(W)})$, 
\begin{equation}
\displaystyle\int_{\widetilde{\Sp(W)}} \Psi(\widetilde{g})T(\widetilde{g})d\widetilde{g}
\label{Schwartz1}
\end{equation}
is in $\S(W)$. Brieffly, for $\Psi \in \mathscr{C}^{\infty}_{c}(\widetilde{\Sp(W)})$ such that $\supp(\Psi) \subseteq \Im(\widetilde{c})$ and $\phi \in \S(W)$, we have:
\begin{eqnarray*}
& &\left(\displaystyle\int_{\widetilde{\Sp(W)}} \Psi(\widetilde{g})T(\widetilde{g})d\widetilde{g}\right)(\phi) = \displaystyle\int_{W} \displaystyle\int_{\mathfrak{sp}(W)} \Psi(\widetilde{c}(X))\Theta(\widetilde{c}(X))\phi(w) j_{\mathfrak{sp}}(X)\chi\left(\frac{1}{4}\langle Xw, w\rangle\right) dX dw\\
& = & \displaystyle\int_{W} \left(\displaystyle\int_{\mathfrak{sp}(W)} \varphi(X) \chi\left(\frac{1}{4}\langle Xw, w\rangle\right) dX\right) \phi(w) dw = \displaystyle\int_{W} \widehat{\varphi} \circ \tau_{\mathfrak{sp}}(w) \phi(w) dw 
\end{eqnarray*}
where $\varphi(X) = \Psi(\widetilde{c}(X)) \Theta(\widetilde{c}(X)) j_{\mathfrak{sp}}(X) \in \mathscr{C}^{\infty}_{c}(\mathfrak{sp}(W))$. Then, $\widehat{\phi} \in \S(\mathfrak{sp}(W))$ and $\lambda(w) = \widehat{\phi} \circ \tau_{\mathfrak{sp}}(w) \in \S(W)$.

\noindent Similarly, for all $\widetilde{p} \in \widetilde{\Sp(W_{\mathbb{C}})^{++}}$ and $\Psi \in \mathscr{C}^{\infty}_{c}(\widetilde{\Sp(W)})$, one can prove that there exists $\lambda_{\widetilde{p}} \in \S(W)$ such that for every $\phi \in \S(W)$, we have:
\begin{equation}
\left(\displaystyle\int_{\widetilde{\Sp(W)}} \Psi(\widetilde{g})T(\widetilde{p}\widetilde{g})d\widetilde{g}\right)(\phi) = \displaystyle\int_{W} \lambda_{\widetilde{p}}(w) \phi(w) dw.
\label{Schwartz1}
\end{equation}
The link between the functions $\lambda_{\widetilde{p}}$ and $\lambda$ is given by the following equality:
\begin{equation}
\lambda_{\widetilde{p}}(w) = T(\widetilde{p}) \natural \lambda(w) \qquad (w \in W).
\label{TwistedDistribution}
\end{equation}   

\begin{lemme}

For all $\widetilde{g} \in \widetilde{\Sp(W)}^{c}$ and $\widetilde{h} \in \widetilde{\Sp(W_{\mathbb{C}})^{++}}$, we get:
\begin{equation*}
C(g, h) \chi_{c(gh)}(w) = \displaystyle\int_{W} \chi_{c(g)}(u) \chi_{c(h)}(w-u) \chi\left(\frac{1}{2}\langle u, w\rangle\right) du \qquad (\forall w \in W).
\end{equation*}

\label{lemma1908}

\end{lemme}

\begin{proof}

We have $T(\widetilde{g}\widetilde{h}) = T(\widetilde{g}) \natural T(\widetilde{h})$, i.e. $T(\widetilde{g}\widetilde{h}) \natural \phi = T(\widetilde{g}) \natural T(\widetilde{h}) \natural \phi$ for all $\phi \in \S(W)$.

\noindent For all $w \in W$, we have:
\begin{eqnarray*}
T(\widetilde{g}\widetilde{h}) \natural \phi(w) & = & \displaystyle\int_{W} \Theta(\widetilde{g}\widetilde{h}) \chi_{c(gh)}(u) \phi(w-u) \chi\left(\frac{1}{2}\langle u, w\rangle\right) du \\
                                     & = & \Theta(\widetilde{g})\Theta(\widetilde{h})C(g, h) \displaystyle\int_{W} \chi_{c(gh)}(u) \phi(w-u) \chi\left(\frac{1}{2}\langle u, w\rangle\right) du \\
                                     & = & -\Theta(\tilde{g}) \Theta(\tilde{h}) C(g, h) \displaystyle\int_{W} \chi_{c(gh)}(w-v) \chi\left(-\frac{1}{2}\langle v, w\rangle\right) dv
\end{eqnarray*}  
and
\begin{eqnarray*}
& & (T(\widetilde{g}) \natural T(\widetilde{h})) \natural \phi(w) = T(\widetilde{g}) \natural (T(\widetilde{h}) \natural \phi)(w) = \displaystyle\int_{W} \Theta(\widetilde{g}) \chi_{c(g)}(u) T(\widetilde{h}) \natural \phi(w-u) \chi\left(\frac{1}{2}\langle u, w\rangle\right) du \\
            & = & \displaystyle\int_{W} \Theta(\widetilde{g}) \chi_{c(g)}(u) \left(\displaystyle\int_{W} \Theta(\widetilde{h}) \chi_{c(h)}(v) \phi(w-u-v) \chi\left(\frac{1}{2}\langle v, w-u\rangle\right)dv\right) \chi\left(\frac{1}{2}<u, w>\right) du \\  
            & = & \displaystyle\int_{W} \Theta(\widetilde{g}) \chi_{c(g)}(u) \left(\displaystyle\int_{W} \Theta(\widetilde{h}) \chi_{c(h)}(w-u-z) \phi(z) \chi\left(-\frac{1}{2}\langle z, w-u\rangle\right) dz \right) \chi\left(\frac{1}{2}\langle u, w\rangle\right) du \\
            & = & \Theta(\widetilde{g}) \Theta(\widetilde{h}) \displaystyle\int_{W} \phi(z) \left(\displaystyle\int_{W} \chi_{c(g)}(u) \chi_{c(h)}(w-u-z) \chi\left(-\frac{1}{2}\langle z, w-u\rangle\right) \chi\left(\frac{1}{2}\langle u, w\rangle\right)du\right)dz \\
            & = & C(g, h) \chi_{c(gh)}(w-v) \chi\left(-\frac{1}{2}\langle v, w\rangle\right)
\end{eqnarray*}
Then, for all $v, w \in W$, we get:
\begin{equation*}
C(g, h) \chi_{c(gh)}(w-v) \chi\left(-\frac{1}{2}\langle v, w\rangle\right) = \displaystyle\int_{W} \chi_{c(g)}(u) \chi_{c(h)}(w-u-v) \chi\left(-\frac{1}{2} \langle v, w-u\rangle\right) \chi\left(\frac{1}{2} \langle u, w\rangle\right)du.
\end{equation*}
We get the result by taking $v = 0$. 
                                  
\end{proof}

\begin{prop}

For every $\widetilde{p} \in \widetilde{\Sp(W_{\mathbb{C}})}^{++}$ and $\Psi, \Phi \in \mathscr{C}^{\infty}_{c}(\widetilde{\Sp(W)})$, we get that:
\begin{equation*}
\displaystyle\int_{\widetilde{\Sp(W)}} \Psi(\widetilde{g}) \displaystyle\int_{\widetilde{\Sp(W)}} \Phi(\widetilde{h}) T(\widetilde{g}\widetilde{h}\widetilde{p}) d\widetilde{h}d\widetilde{g}
\end{equation*}
is a Schwartz function $\phi_{\widetilde{p}}$ given by:
\begin{equation*}
\phi_{\widetilde{p}}(w) = \displaystyle\int_{\widetilde{\Sp(W)}} \Psi(\widetilde{g}) \displaystyle\int_{\widetilde{\Sp(W)}} \Phi(\widetilde{h}) \Theta(\widetilde{g}\widetilde{h}\widetilde{p}) \chi_{c(ghp)}(w) d\widetilde{h}d\widetilde{g}.
\end{equation*}

\label{PropCocycle}

\end{prop}

\begin{proof}

For all $\phi \in \S(W)$, we have:
\begin{eqnarray*}
& & \left(\displaystyle\int_{\widetilde{\Sp(W)}} \Psi(\widetilde{g}) \displaystyle\int_{\widetilde{\Sp(W)}} \Phi(\widetilde{h}) T(\widetilde{g}\widetilde{h}\widetilde{p}) d\widetilde{h}d\widetilde{g}\right)(\phi) = \displaystyle\int_{\widetilde{\Sp(W)}} \Psi(\widetilde{g}) \displaystyle\int_{\widetilde{\Sp(W)}} \Phi(\widetilde{h}) T(\widetilde{g}\widetilde{h}\widetilde{p}) \natural \phi(0) d\widetilde{h}d\widetilde{g} \\
& = & \displaystyle\int_{\widetilde{\Sp(W)}} \Psi(\widetilde{g}) T(\widetilde{g}) \natural \left(\displaystyle\int_{\widetilde{\Sp(W)}} \Phi(\widetilde{h}) T(\widetilde{h}\widetilde{p}) \natural \phi d\widetilde{h}\right)(0) d\widetilde{g} \\
& = & \displaystyle\int_{\widetilde{\Sp(W)}} \Psi(\widetilde{g}) \displaystyle\int_{W} \Theta(\widetilde{g}) \chi_{c(g)}(w) \left(\displaystyle\int_{\widetilde{\Sp(W)}} \Phi(\widetilde{h})T(\widetilde{h}\widetilde{p}) \natural \phi d\widetilde{h}\right)(-w) dw d\widetilde{g} \\
& = & \displaystyle\int_{\widetilde{\Sp(W)}} \Psi(\widetilde{g}) \displaystyle\int_{W} \Theta(\widetilde{g}) \chi_{c(g)}(w) \displaystyle\int_{\widetilde{\Sp(W)}} \Phi(\widetilde{h})\left(\displaystyle\int_{W} \Theta(\widetilde{h}\widetilde{p}) \chi_{c(hp)}(u) \phi(-w-u) \chi\left(-\frac{1}{2}\langle u, w\rangle\right) du\right) d\widetilde{h} dw d\widetilde{g} \\
& = & \displaystyle\int_{W} \phi(v) \left(\displaystyle\int_{\widetilde{\Sp(W)}} \displaystyle\int_{\widetilde{\Sp(W)}} \displaystyle\int_{W} \Psi(\widetilde{g}) \Phi(\widetilde{h}) \Theta(\widetilde{g})\Theta(\widetilde{h} \widetilde{p}) \chi_{c(g)}(u-v) \chi_{c(hp)}(u) \chi\left(\frac{1}{2}\langle u, w\rangle\right)du d\widetilde{h}d\widetilde{g}\right) dv \\
& = & \displaystyle\int_{W} \phi(v) \left(\displaystyle\int_{\widetilde{\Sp(W)}} \displaystyle\int_{\widetilde{\Sp(W)}} \Psi(\widetilde{g}) \Phi(\widetilde{h}) \Theta(\widetilde{g})\Theta(\widetilde{h} \widetilde{p}) C(g, hp) \chi_{chp}(v)d\widetilde{h}d\widetilde{g}\right) dv \\
& = & \displaystyle\int_{W} \phi(v) \left(\displaystyle\int_{\widetilde{\Sp(W)}} \displaystyle\int_{\widetilde{\Sp(W)}} \Psi(\widetilde{g}) \Phi(\widetilde{h}) \Theta(\widetilde{g}\widetilde{h}\widetilde{p}) \chi_{c(ghp)}(v)d\widetilde{g}d\widetilde{h}\right) dv
\end{eqnarray*}
(where the last equality is obtained using Lemma \eqref{lemma1908}).

\end{proof}

\noindent Now, we are able to prove state and prove the following theorem.

\begin{theo}

For every function $\Psi \in \mathscr{C}^{\infty}_{c}(\widetilde{\G}')$, we get:
\begin{equation*}
\Theta_{\Pi'}(\Psi) = \lim\limits_{\underset{\widetilde{p} \in \widetilde{\Sp(W_{\mathbb{C}})}^{++}}{\widetilde{p} \to 1}} \displaystyle\int_{\widetilde{\G}'} \displaystyle\int_{\widetilde{\G}} \overline{\Theta_{\Pi}(\widetilde{g})} \Theta(\widetilde{g} \widetilde{g}'\widetilde{p}) \Psi(\widetilde{g}') d\tilde{g} d\tilde{g'}.
\end{equation*}
Then, as a distributions on $\widetilde{\G}'$, we have:
\begin{equation}
\Theta_{\Pi'}(\widetilde{g}') = \lim\limits_{\underset{\widetilde{p} \in \widetilde{\Sp(W_{\mathbb{C}})}^{++}}{\widetilde{p} \to 1}} \displaystyle\int_{\widetilde{\G}} \overline{\Theta_{\Pi}(\widetilde{g})} \Theta(\widetilde{g} \widetilde{g}' \widetilde{p}) d\tilde{g}.
\label{TheoLimiteSemigroup}
\end{equation}

\label{TheoCharacter}
\end{theo}

\begin{proof}

According to Proposition \ref{PropCocycle}, there exists a function $\lambda_{\widetilde{p}} \in \S(W)$ such that
\begin{equation*}
\left(\displaystyle\int_{\widetilde{\G}'} \displaystyle\int_{\widetilde{\G}} \overline{\Theta_{\Pi}(\widetilde{g})} T(\widetilde{g} \widetilde{g}'\widetilde{p}) \Psi(\widetilde{g}') d\tilde{g} d\tilde{g'}\right)(\phi) = \displaystyle\int_{W} \lambda_{\widetilde{p}}(w) \phi(w) dw \qquad (\phi \in \S(W)).
\end{equation*}
Similarly, there exists $\lambda \in \S(W)$ such that
\begin{equation*}
\left(\displaystyle\int_{\widetilde{\G}'} \displaystyle\int_{\widetilde{\G}} \overline{\Theta_{\Pi}(\widetilde{g})} T(\widetilde{g} \widetilde{g}') \Psi(\widetilde{g}') d\tilde{g} d\tilde{g'}\right)(\phi) = \displaystyle\int_{W} \lambda(w) \phi(w) dw \qquad (\phi \in \S(W)).
\end{equation*}
Using Equation \eqref{TwistedDistribution}, for all $w \in W$, we have $\lambda_{\widetilde{p}}(w) = T(\widetilde{p}) \natural \lambda(w)$.

\noindent We have that $\delta_{0} = T(\widetilde{1}) = \lim\limits_{\underset{\widetilde{p} \in \widetilde{\Sp(W_{\mathbb{C}})}^{++}}{\widetilde{p} \to 1}} T(\widetilde{p})$, and then, using \cite[Section~4.5]{TOM1},
\begin{equation*}
\lim\limits_{\underset{\widetilde{p} \in \widetilde{\Sp(W_{\mathbb{C}})}^{++}}{\widetilde{p} \to 1}} \lambda_{\widetilde{p}}(0) = \lim\limits_{\underset{\widetilde{p} \in \widetilde{\Sp(W_{\mathbb{C}})}^{++}}{\widetilde{p} \to 1}} T(\widetilde{p}) \natural \lambda(0) = \delta_{0} \natural \lambda(0) = \lambda(0).
\end{equation*}
Then, using \cite[Theorem~3.5.4]{HOW4}, we get:
\begin{eqnarray*}
\Theta_{\Pi'}(\Psi) & = & \tr \displaystyle\int_{\widetilde{\G}'} \displaystyle\int_{\widetilde{\G}} \overline{\Theta_{\Pi}(\widetilde{g})} \Psi(\widetilde{g}') \omega(\widetilde{g} \widetilde{g}')d\tilde{g} d\tilde{g}' \\
                              & = &\tr \Op \circ \mathscr{K} \displaystyle\int_{\widetilde{\G}'} \displaystyle\int_{\widetilde{\G}} \overline{\Theta_{\Pi}(\widetilde{g})} \Psi(\widetilde{g}') T(\widetilde{g} \widetilde{g}')d\tilde{g} d\tilde{g'} \\
                             & = & \left( \displaystyle\int_{\widetilde{\G}'} \displaystyle\int_{\widetilde{\G}} \overline{\Theta_{\Pi}(\widetilde{g})} \Psi(\widetilde{g}') T(\widetilde{g} \widetilde{g}')d\tilde{g} d\tilde{g'} \right)(0) = \lambda(0) \\
                             & = & \lim\limits_{\underset{\widetilde{p} \in \widetilde{\Sp(W_{\mathbb{C}})}^{++}}{\widetilde{p} \to 1}} \lambda_{\widetilde{p}}(0) = \lim\limits_{\underset{\widetilde{p} \in \widetilde{\Sp(W_{\mathbb{C}})}^{++}}{\widetilde{p} \to 1}} \displaystyle\int_{\widetilde{\G}'} \displaystyle\int_{\widetilde{\G}} \overline{\Theta_{\Pi}(\widetilde{g})} \Theta(\widetilde{g} \widetilde{g}'\widetilde{p}) \Psi(\widetilde{g}') d\tilde{g} d\tilde{g}'.
\end{eqnarray*}                             
                  
\end{proof}

\noindent From now on, we assume that $\G$ is connected. For every $\widetilde{p} \in \widetilde{\Sp(W_{\mathbb{C}})}^{++}$ and $\widetilde{g}' \in \widetilde{\G}'$, we define the function $\F_{\widetilde{p}, \widetilde{g}'}: \widetilde{\G} \to \mathbb{C}$ by:
\begin{equation*}
\F_{\widetilde{p}, \widetilde{g}'}(\widetilde{g}) = \overline{\Theta_{\Pi}(\widetilde{g})} \Theta(\widetilde{g} \widetilde{g}' \widetilde{p}).
\end{equation*}
We easily prove that for every element $g \in \G$, we have:
\begin{equation*}
\F_{\widetilde{p}, \widetilde{g}'}((g, \xi)) = \F_{\widetilde{p}, \widetilde{g}'}((g, -\xi))
\end{equation*}
and by a standard result of differential geometry (see \cite[Lemma~A.4.2.11]{WAL2}), we get
\begin{equation*}
\displaystyle\int_{\widetilde{\G}} \F_{\widetilde{p}, \widetilde{g}'}(\widetilde{g}) d\tilde{g} = 2 \displaystyle\int_{\G} \H_{\widetilde{p}, \widetilde{g}'}(g) dg,                                   
\end{equation*}
where $dg$ is the normalized Haar measure on $\G$ and $\H_{\widetilde{p}, \widetilde{g}'}: \G \to \mathbb{C}$ is the function defined by:
\begin{equation*}
\H_{\widetilde{p}, \widetilde{g}'}(\pr(\widetilde{g})) = \F_{\widetilde{p}, \widetilde{g}'}(\widetilde{g}) \qquad (\widetilde{g} \in \widetilde{\G}).
\end{equation*}
From now on, we assume that $\G$ is connected. By Weyl's integration formula (see \cite[Theorem~8.60]{KNA}), we get:
\begin{equation*}
\displaystyle\int_{\G} \H_{\widetilde{p}, \widetilde{g}'}(g) dg = \displaystyle\int_{\T} \left(\displaystyle\int_{\G/\T} \H_{\widetilde{p}, \widetilde{g}'}(gtg^{-1}) dg\T \right) |D(t)|^{2} dt
\end{equation*}
where $D$ is the Weyl denominator. We define $\G'^{++} = \G'_{\mathbb{C}} \cap \Sp(W_{\mathbb{C}})^{++}$ and denote by $\widetilde{\G'}^{++}$ the preimage in $\widetilde{\Sp(W_{\mathbb{C}})}^{++}$. For every element $\widetilde{p} \in \widetilde{\G}'^{++}$, we prove easily that the function $\H_{\widetilde{p}, \widetilde{g}'}$ in invariant by conjugation. In particular, we get:
\begin{equation*}
\displaystyle\int_{\T} \left(\displaystyle\int_{\G/\T} \H_{\widetilde{p}, \widetilde{g}'}(gtg^{-1}) dg\T \right) |D(t)|^{2} dt = \displaystyle\int_{\T} \H_{\widetilde{p}, \widetilde{g}'}(t) |D(t)|^{2} dt
\end{equation*}
Using Theorem \ref{TheoCharacter}, we get:

\begin{lemme}

For every regular element $\widetilde{g}' \in \widetilde{\G}'$, the character $\Theta_{\Pi'}$ of $\Pi'$ is given by the following formula:
\begin{equation}
\Theta_{\Pi'}(\widetilde{g}') = \lim\limits_{\underset{\widetilde{p} \in \widetilde{\G}'^{++}}{\widetilde{p} \to 1}} \displaystyle\int_{\T} \H_{\widetilde{p}, \widetilde{g}'}(t) |\Delta(t)|^{2} dt.
\label{FormuleIntegrale1908}
\end{equation}

\end{lemme}

\noindent Using an article of M. Kashiwara and M. Vergne \cite{VER}, we obtain the weights of the representations $\Pi \in \widehat{\widetilde{\G}}_{\omega}$. Then, using the Weyl character formula (Equation \eqref{WeylCharacter}), we get a formula for the character $\Theta_{\Pi}$. What we need now, is an explicit realisation of the character $\Theta$ of the metaplectic representation.

\section{A restriction of $\Theta$ to a maximal compact subgroup}

\noindent We recall here the main ideas of \cite[Section~2]{TOM8}. Here we want an explicit formula for the character $\Theta$ of the metaplectic representation on a maximal compact subgroup of $\Sp(W)$. We know that for any positive complex structure $J$ on $W$, the subgroup $\Sp(W)^{J}$ of symplectic matrices which commute with $J$ is a maximal compact subgroup of $\Sp(W)$. More precisely, for every compact dual pair $(\G, \G')$ (with $\G$ compact), there exists a complex structure $J$ of $\Sp(W)$ such that $\G.\T' \subseteq \Sp(W)^{J}$, where $\T'$ is the maximal compact Cartan subgroup of $\G'$ (we will construct this element $J$ explicitly for the dual pair $(\U(n, \mathbb{C}), \U(p, q, \mathbb{C}))$ in Section \ref{ComputationsUnitary}).

\noindent We fix a positive complex structure $J$ on $W$, and we denote by $W_{\mathbb{C}}$ the complexification of $W$. With respect to the endomorphism $J$, we get a decomposition of $W_{\mathbb{C}}$ of the form
\begin{equation*}
W_{\mathbb{C}} = W^{+}_{\mathbb{C}} \oplus W^{-}_{\mathbb{C}}
\end{equation*}
where $W^{+}_{\mathbb{C}}$ (resp. $W^{-}_{\mathbb{C}}$) is the $i$-eigenspace (resp. $-i$-eigenspace) for $J$. One can prove easily that the restriction of the form $\H$ defined in Equation \eqref{HermitianForm} to the space $W^{+}_{\mathbb{C}}$ is positive definite. We denote by $\U = \U(W^{+}_{\mathbb{C}}, \H_{|_{W_{\mathbb{C}}}})$ the subgroup of $\GL(W^{+}_{\mathbb{C}})$ which preserve the form $\H_{|_{W_{\mathbb{C}}}}$.

\noindent We define a two fold cover of $\U$, denoted by $\widetilde{\U}$, as
\begin{equation}
\widetilde{\U} = \left\{(u, \xi), \xi^{2} = \det(u), u \in \U\right\} \subseteq \GL(W^{+}_{\mathbb{C}}) \times \mathbb{C}^{*}.
\label{UnitaryCover}
\end{equation}
Then, $\widetilde{\U}$ is a group (endowed with the pointwise multiplication). More precisely, it's a connected two-fold covering of $\U$.

\begin{prop}

The map:
\begin{equation*}
\Sp(W)^{J} \ni g \to g_{|_{W^{+}_{\mathbb{C}}}} \in \U
\end{equation*}
is a group isomorphism and lifts to an isomorphism 
\begin{equation*}
\widetilde{\Sp(W)}^{J} \ni (g, \xi) \to (u, \xi\det(g-1)_{|_{(g-1)W^{+}_{\mathbb{C}}}}) \in \widetilde{\U}.
\end{equation*}
Then, the restriction of the metaplectic cover to $\Sp(W)^{J}$ is isomorphic to the covering 
\begin{equation*}
\widetilde{\U} \ni (u, \xi) \to u \in \U
\end{equation*}

\end{prop}

\begin{proof}

The proof of this result can be found in \cite[Proposition~1]{TOM8}.

\end{proof}

\noindent According to Equation \eqref{TheoLimiteSemigroup}, we need a formula for $\Theta$ not only on $\Sp(W)^{J}$, but on an analogue subset in the oscillator semigroup. Briefly, the map
\begin{equation*}
(\Sp(W_{\mathbb{C}})^{++})^{J} \ni g \to g_{|_{W^{+}_{\mathbb{C}}}} \in \GL(W^{+}_{\mathbb{C}})
\end{equation*}
is well define and bijective. We now define a subgroup $\GL(W^{+}_{\mathbb{C}})^{++}$ of $\GL(W^{+}_{\mathbb{C}})$ as
\begin{equation*}
\GL(W^{+}_{\mathbb{C}})^{++} = \left\{h \in \GL(W^{+}_{\mathbb{C}}), \H_{|_{W_{\mathbb{C}}}}(w, w) > \H_{|_{W_{\mathbb{C}}}}(hw, hw), 0 \neq w \in W^{+}_{\mathbb{C}}\right\}
\end{equation*}
As in Equation \eqref{UnitaryCover}, we define a non-trivial double cover of $\GL(W^{+}_{\mathbb{C}})^{++}$ by
\begin{equation*}
\widetilde{\GL(W^{+}_{\mathbb{C}})}^{++} = \left\{(h, \xi) \in \GL(W^{+}_{\mathbb{C}})^{++} \times \mathbb{C}^{*}, \xi^{2} = \det(h)\right\}.
\end{equation*}
The group structure on $\widetilde{\GL(W^{+}_{\mathbb{C}})}^{++}$ is given by the coordinate-wise multiplication. More particularly, we get the following proposition.

\begin{prop}

The set $\widetilde{\GL(W^{+}_{\mathbb{C}})}^{++} \cup \widetilde{\U}$ is a semigroup. Moreover, the map
\begin{equation*}
(\Sp(W_{\mathbb{C}})^{++})^{J} \cup \widetilde{\Sp(W)}^{J} \ni g \to g_{|_{W^{+}_{\mathbb{C}}}} \in  \widetilde{\GL(W^{+}_{\mathbb{C}})}^{++} \cup \widetilde{\U}
\end{equation*}
is a semigroup isomorphism.

\end{prop}

\noindent The following corollary gives us the character $\Theta$ on the subsemigroup $(\Sp(W_{\mathbb{C}})^{++})^{J} \cup \widetilde{\Sp(W)}^{J}$.

\begin{coro}

The restriction of the character $\Theta$ on the subsemigroup $(\widetilde{\Sp(W^{+}_{\mathbb{C}})}^{++})^{J} \cup \widetilde{\Sp(W)}^{J}$ is given by
\begin{equation}
\Theta(\widetilde{k}) = \lim\limits_{\underset{\widetilde{h} \in \widetilde{\GL(W^{+}_{\mathbb{C}})}^{++}}{\widetilde{h} \to \widetilde{k}}} \Theta(\widetilde{h}) = \lim\limits_{\underset{\widetilde{h} \in \widetilde{\GL(W^{+}_{\mathbb{C}})}^{++}}{\widetilde{h} \to \widetilde{k}}} \cfrac{\xi}{\det(1-h)} \qquad (\widetilde{h} = (h, \xi)).
\label{Identification}
\end{equation}

\label{CorollarySemigroupCharacter}
\end{coro}

\section{The dual pair $(\G = \U(n, \mathbb{C}), \G' = \U(p, q, \mathbb{C}))$}

\label{ComputationsUnitary}

\noindent Let $(V, b)$ be a $n$-dimensional vector space over $\mathbb{C}$ endowed with a positive definite hermitian form $b$ and $\mathscr{B}$ be a basis of $V$ such that $\Mat(b, \mathscr{B}) = \Id_{n}$. We denote by $\U(V, b)$ the group of isometries of $b$, i.e.
\begin{equation}
\U(V, b) = \left\{g \in \GL(V), b(gu, gv) = b(u, v), (\forall u, v \in V)\right\}.
\label{UnitaireU(n)}
\end{equation}
By writing the endomorphisms in the basis $\mathscr{B}$, we get that the right hand side of Equation \eqref{UnitaireU(n)} can be written as:
\begin{equation}
\left\{g \in \GL(n, \mathbb{C}), g^{*}g = \Id_{n}\right\},
\label{UnitaireU(n)2}
\end{equation}
where $g^{*} = {g^{-1}}^{t}$. We denote by $\G = \U(n, \mathbb{C})$ the group defined in Equation \eqref{UnitaireU(n)2}, by $\mathfrak{g} = \mathfrak{u}(n, \mathbb{C})$ the Lie algebra of $\U(n, \mathbb{C})$. The maximal torus $\T$ of $\U(n, \mathbb{C})$ is given by $\T = \left\{\diag(t_{1}, \ldots, t_{n}), t_{i} \in \S^{1}\right\}$ and its Lie algebra $\mathfrak{t}$ is defined as:
\begin{equation*}
\mathfrak{t} = \bigoplus\limits_{k=1}^{n} i\mathbb{R}\E_{k, k}.
\end{equation*}
One can check (see \cite[Chapter~II]{KNA}) that the roots of $\mathfrak{g}_{\mathbb{C}}$ with respect to $\mathfrak{t}_{\mathbb{C}}$ are given by:
\begin{equation*}
\Phi(\mathfrak{g}_{\mathbb{C}}, \mathfrak{t}_{\mathbb{C}}) = \left\{\pm(e_{i} - e_{j}), 1 \leq i < j \leq n\right\},
\end{equation*}
where $e_{k}(\diag(h_{1}, \ldots, h_{n})) = h_{k}$.

\noindent Similarly, let $(V', b')$ be a $p+q$-dimensional vector space over $\mathbb{C}$ endowed with a non-degenerate hermitian form $b'$ of signature $(p, q)$ and let $\mathscr{B}'$ be a basis of $V'$ such that $\Mat(b', \mathscr{B}') = \Id_{p, q}$. We denote by $\U(V', b')$ the group of isometries of $b'$, i.e.
\begin{equation}
\U(V', b') = \left\{g \in \GL(V'), b'(gu, gv) = b'(u, v), (\forall u, v \in V')\right\},
\label{UnitaireU(n)11}
\end{equation}
and by $\U(p, q, \mathbb{C})$ the following group
\begin{equation}
\left\{g \in \GL(p, q, \mathbb{C}), g^{*}\Id_{p, q}g = \Id_{p, q}\right\}.
\label{UnitaireU(n)22}
\end{equation}
Let $\K' = \U(p, \mathbb{C}) \times \U(q, \mathbb{C})$ be the maximal compact subgroup of $\G'$. 

\noindent Using the paper of M. Kashiwara and M. Vergne \cite{VER} (we can also use the Appendix of \cite{TOM9}), the weights of the representations of $\Pi \in \widehat{\widetilde{\U(n,\mathbb{C})}}_{\omega}$ which appears in the correspondence are given by the following formula:
\begin{equation}
\lambda = \sum\limits_{a=1}^n \cfrac{q - p}{2} e_{a} - \sum\limits_{a = 1}^r \nu_{a} e_{n+1-a} + \sum\limits_{a=1}^s \mu_{a} e_{a},
\label{Poids 1}
\end{equation}
where $0 \leqslant r \leqslant p$, $0 \leqslant s \leqslant q$, $r+s \leqslant m$, and integers $\nu_{1}, \ldots, \nu_{r}, \mu_{1}, \ldots, \mu_{s}$ which satisfy $\nu_{1} \geqslant \ldots \geqslant \nu_{r} > 0$ and $\mu_{1} \geqslant \ldots \geqslant \mu_{s} > 0$. The weights $\lambda$ can also be written as
\begin{equation}
\lambda = \sum\limits_{a=1}^m \left(\cfrac{q - p}{2} + \lambda_{a}\right) e_{a}
\label{Weights}
\end{equation}
where $\lambda_{i} \in \mathbb{Z}$, $\lambda_{1} \geq \ldots \geq \lambda_{m}$ with at most $q$ of the integers $\lambda_{i}$ are positives and $p$ negatives.

\noindent We easily proved that, for $\G = \U(n, \mathbb{C})$, we have:
\begin{equation*}
\rho = \frac{1}{2} \sum\limits_{\alpha \in \Phi^{+}(\mathfrak{g}_{\mathbb{C}}, \mathfrak{t}_{\mathbb{C}})} \alpha = \sum\limits_{a = 1}^{n} \cfrac{n-2a+1}{2} e_{a}
\end{equation*}

\bigskip

\noindent Using Corollary \ref{CorollarySemigroupCharacter}, we give a formula for the character $\Theta$ on $\T. \T'^{++}$, where $\T$ and $\T'$ are diagonal Cartan subgroups of $\G$ and $\G'$ respectively and $\T'^{++} = \T'_{\mathbb{C}} \cap \Sp(W_{\mathbb{C}})^{++}$.

\begin{prop}
\begin{enumerate}
\item The set $\T'^{++}$ is given by 
\begin{equation}
\T'^{++} = \left\{\diag(t_{1}, \ldots, t_{p+q}) \thinspace ; \thinspace \thinspace |t_{i}| < 1 \text{ for } 1 \leq i \leq p, |t_{i}| > 1 \text{ for } p < i \leq p+q\right\}.
\label{Semi}
\end{equation}
\item For all $\widetilde{t} \in \widetilde{\T}$ and $\widetilde{t}' \in \T'^{++}$, the character $\Theta$ is given by:
\begin{equation}
\Theta(\widetilde{t}\widetilde{t}') = \cfrac{(-1)^{mq}\prod\limits_{b = 1}^m t^{p}_{b}\left( \prod\limits_{a=1}^{p+q} t'_{a} \right)^{\frac{m}{2}} \left(\prod\limits_{b=1}^m t^{\frac{q-p}{2}}_{b}\right)}{\prod\limits_{a=1}^{p} \prod\limits_{b=1}^m (t_{b} - t'_{a}) \prod\limits_{a=p+1}^{p+q} \prod\limits_{b=1}^m \left(t_{b} - \frac{1}{\overline{t'_{a}}}\right)}.
\label{Caractere sur le tore}
\end{equation}
\end{enumerate}

\label{Proposition19081990}

\end{prop}

\begin{proof}

\begin{enumerate}
\item See Appendix \ref{AppendixA}.
\item We consider $W = \M((p+q) \times m, \mathbb{C})$ as a real vector space endowed with the following form
\begin{equation*}
\langle w, w'\rangle  = \Im(\bar{w'}^{t}I_{p, q}w)
\end{equation*}
This form is symmetric and non-degenerate. Moreover, the map $J(w) = iI_{p, q}w$ is a posiitive definite complex structure on $W$. The maps
\begin{equation*}
\G \times W \ni (g, X) \to gX \in W \qquad \G' \times W \ni (g', X) = Xg'^{-1}
\end{equation*}
give embeddings of $\G$ and $\G'$ into $\Sp(W)$. For every matrix $\E_{a, b} \in W$, we have:
\begin{equation*}
(tt')\E_{a, b} = t'_{a}t^{-1}_{b} \E_{a, b}
\end{equation*}
By definition of $J$, we get:
\begin{equation*}
J(\E_{a, b}) = \begin{cases} i\E_{a, b} & \text{ if } 1 \leq a \leq p \\ -i\E_{a, b} & \text{ if } a > p \end{cases}
\end{equation*}
Then, the eigenvalues of $tt'$ are of the form
\begin{equation*}
\left\{t'_{a}t^{-1}_{b} \thinspace ; \thinspace 1 \leq a \leq p, 1 \leq b \leq n\right\} \cup \left\{\overline{t'_{a}t^{-1}_{b}} \thinspace ; \thinspace p+1 \leq a \leq p+q, 1 \leq b \leq n\right\}.
\end{equation*}
Finally, using Equation \eqref{Identification}, we get:
\begin{equation*}
\Theta(\widetilde{tt^{'}}) = \cfrac{\left( \prod\limits_{a=1}^p \prod\limits_{b=1}^n t'_{a} t^{-1}_{b} \prod\limits_{a=p+1}^{p+q} \prod\limits_{b=1}^n \overline{t'_{a}t^{-1}_{b}} \right)^{\frac{1}{2}}}{\prod\limits_{a=1}^p \prod\limits_{b=1}^n (1 - t'_{a} t^{-1}_{b}) \prod\limits_{a=p+1}^{p+q} \prod\limits_{b=1}^n (1 - \overline{t'_{a} t^{-1}_{b}})} = \cfrac{\left( \prod\limits_{a=1}^p (t'_{a})^{n} \prod\limits_{a=p+1}^{p+q} (\overline{t'_{a}})^{n} \prod\limits_{b=1}^n t^{q-p}_{b} \right)^{\frac{1}{2}}}{\prod\limits_{a=1}^p \prod\limits_{b=1}^n (1 - t'_{a} t^{-1}_{b}) \prod\limits_{a=p+1}^{p+q} \prod\limits_{b=1}^n (1 - \overline{t'_{a} t^{-1}_{b}})} 
\end{equation*}
In particular, for every element $\widetilde{t}'$ of $\widetilde{\T}'$, we get:
\begin{equation*}
\Theta(\widetilde{t^{'}}) = \cfrac{\left( \prod\limits_{a=1}^p (t'_{a})^{m} \prod\limits_{a=p+1}^{p+q} (\overline{t'_{a}})^{m}\right)^{\frac{1}{2}}}{\prod\limits_{a=1}^p \prod\limits_{b=1}^m (1 - t'_{a}) \prod\limits_{a=p+1}^{p+q} \prod\limits_{b=1}^m (1 - \overline{t'_{a}})}
\end{equation*}
According to Equation \eqref{Equation Theta}, we get that:
\begin{equation*}
\Theta(\widetilde{t}\widetilde{t}') = \Theta(\widetilde{t})\Theta(\widetilde{t}') \Lambda(c(t)+c(t')).
\end{equation*}
The rest of the proof is a straightforward computation.
\end{enumerate}

\end{proof}

\begin{prop}

For every regular element $\widetilde{t}'$ in $\widetilde{\T}'$, we get:
\begin{eqnarray*}
\Theta_{\Pi'}(\widetilde{t}') & = &  \cfrac{(-1)^{nq +\frac{(n-1)n}{2}} \left(\prod\limits_{a=1}^{p+q} t'_{a}\right)^{\frac{n}{2}}}{(2i\pi)^{n}n!}\lim\limits_{\underset{0 < r < 1}{r \to 1}} \sum\limits_{w \in \mathscr{W}} \sgn(w) \displaystyle\int_{\S^{1}} \ldots \displaystyle\int_{\S^{1}} \cfrac{\prod\limits_{b = 1}^n t^{p-n-1}_{b} \prod\limits_{a = 1}^{n} t^{a - \lambda_{a}}_{w(a)} \prod\limits_{1 \leq i < j \leq n}(t_{i} - t_{j})}{\prod\limits_{a=1}^p \prod\limits_{b=1}^n (t_{b} - rt'_{a}) \prod\limits_{a=p+1}^{p+q} \prod\limits_{b=1}^n \left(t_{b} - \frac{1}{r\overline{t'_{a}}}\right)} \prod\limits_{k = 1}^n dt_{k} \\
& = &K(t')\lim\limits_{\underset{0 < r < 1}{r \to 1}} \sum\limits_{w \in \mathscr{W}} \sum\limits_{\beta \in \mathscr{S}_{n}}\sgn(w) \sgn(\beta) \prod\limits_{b = 1}^n \displaystyle\int_{\S^{1}} \cfrac{t^{p-n-2 + w^{-1}(b) - \lambda_{w^{-1}(b)} + \beta^{-1}(b)}_{b} }{\prod\limits_{a=1}^p (t_{b} - rt'_{a}) \prod\limits_{a=p+1}^{p+q} \left(t_{b} - \frac{1}{r\overline{t'_{a}}}\right)} dt_{b}
\label{Formule caractere cartan compact}
\end{eqnarray*}
where $K(t') = \cfrac{(-1)^{nq + \frac{(n-1)n}{2}}\left( \prod\limits_{a=1}^{p+q} t'_{a}\right)^{\frac{n}{2}}}{(2i\pi)^{n}n!}$.
\label{PropositionU(n,C)}
\end{prop}

\begin{proof}
Using Equation \eqref{FormuleIntegrale1908}, we get that
\begin{equation*}
\Theta_{\Pi'}(\widetilde{g}') = \lim\limits_{\underset{\widetilde{p} \in \widetilde{\G}'^{++}_{\mathbb{C}}}{\widetilde{p} \to 1}} \displaystyle\int_{\T} \H_{\widetilde{p}, \widetilde{g}'}(t) |\Delta(t)|^{2} d\mu_{\T}(t).
\end{equation*}
The torus $\T$ of $\U(n, \mathbb{C})$ is isomorphic to ${\S^{1}}^{\otimes n}$. Under this identification, we get:
\begin{equation*}
d\mu_{\T} = \bigotimes_{i=1}^{n} d\mu_{\S^{1}},
\end{equation*}
and $d\mu_{\S^{1}}(z) = \frac{dz}{2i\pi z}$.
Moreover, 
\begin{equation}
D(\widetilde{t} = \widetilde{\exp(x)}) = \prod\limits_{\alpha > 0} (e^{\frac{\alpha}{2}(x)} - e^{-\frac{\alpha}{2}(x)}) = \prod\limits_{1 \leq i < j \leq m} (t^{\frac{1}{2}}_{i}t^{-\frac{1}{2}}_{j} - t^{-\frac{1}{2}}_{i}t^{\frac{1}{2}}_{j}) = \prod\limits_{i=1}^m t^{-\frac{m-1}{2}}_{i} \prod\limits_{1 \leq i < j \leq m} (t_{i} - t_{j}).
\label{WeylDenominator}
\end{equation}
and by using the Vandermonde's determinant formula, we get:
\begin{equation}
\prod\limits_{1 \leq i < j \leq n} (t_{j} - t_{i}) = \sum\limits_{\beta \in \mathscr{S}_{n}} \sgn(\beta) \prod\limits_{i=1}^{n} t^{i-1}_{\beta_{i}}.
\label{Vandermonde Determinant Formula}
\end{equation}

\noindent The rest of the proof is a straightforward computation using Equation \eqref{WeylCharacter}, Proposition \ref{Proposition19081990} and Equation \eqref{Weights}.

\end{proof}

\noindent We now give a technical lemma concerning the integrals which appears in the previous proposition (the proof is obvious using residue theorem).

\begin{lemme}

Let $a_{1}, \ldots, a_{p}$ be $p$-complex numbers such that $|a_{i}| < 1$ for all $i \in [|1, p|]$. Similarly, we consider $a_{p+1}, \ldots, a_{p+q} \in \mathbb{C}$ such that $|a_{i}| > 1$ for all $i \in [|p+1, p+q|]$. Moreover, we assume that $a_{i} \neq a_{j}, i \neq j$. Then, we get:
\begin{equation*}
\frac{1}{2i\pi}\displaystyle\int_{\S^{1}} \cfrac{t^{k}}{\prod\limits_{i=1}^{p+q} (t - a_{i})} dt = \begin{cases}  \sum\limits_{h=1}^p \frac{a^{k}_{h}}{\prod\limits_{j \neq h} (a_{h}-a_{j})}  & \text{ if } k \geqslant 0 \\  - \sum\limits_{h=p+1}^{p+q} \frac{a^{k}_{h}}{\prod\limits_{j \neq h} (a_{h}-a_{j})} & \text { otherwise } \end{cases}
\end{equation*}
\label{Integrale Formule Lemme}
\end{lemme}

\noindent Let's now fix $n = 1$. In this case, the weights $\lambda$ of the representations $\Pi$ are of the form $\lambda = ke_{1}$ if $k$ is positive and $\lambda = ke_{2}$ otherwise. The reason why we voluntarily change the notations here is because the set of irreducible genuine representations of $\widetilde{\U(1, \mathbb{C})}$ is isomorphic to the unitary dual of $\U(1, \mathbb{C})$, which is isomorphic to $\mathbb{Z}$ via the isomorphism:
\begin{equation*}
\mathbb{Z} \ni k \to (x \to e^{2i\pi kx}) \in \widehat{\widetilde{\U(1, \mathbb{C})}}.
\end{equation*}

\noindent Using Proposition \ref{PropositionU(n,C)} and Lemma \ref{Integrale Formule Lemme}, we get the following proposition.

\begin{prop}

The character $\Theta_{\Pi'_{k}}$ of the representation $\Pi'_{k}$ of $\U(p, q, \mathbb{C})$ is given, for every $\widetilde{t}' \in \widetilde{\T'}$, by:
\begin{equation}
\Theta_{\Pi'_{k}}(\widetilde{t}') = \begin{cases} \prod\limits_{i=1}^{p+q} t^{\frac{1}{2}}_{i} \sum\limits_{h=1}^{p} \cfrac{t^{p-(k+1)}_{h}}{\prod\limits_{h \neq j} (t_{h} - t_{j})} \quad & \text { if } k \leqslant 0 \\
- \prod\limits_{i=1}^{p+q} t^{\frac{1}{2}}_{i} \sum\limits_{h = p+1}^{p+q} \cfrac{t^{p-(k+1)}_{h}}{\prod\limits_{h \neq j} (t_{h} - t_{j})} \quad & \text{ otherwise }
\end{cases}
\label{U(1,1,C)}
\end{equation}
\label{CharacterU(1,C)}
\end{prop}

\begin{nota}

The Weyl group $\mathscr{W}$ (resp. $\mathscr{W}(\mathfrak{k})$) of $\G'$ (resp. $\K'$) is isomorphic to $\mathscr{S}_{p+q}$ (resp. $\mathscr{S}_{p} \times \mathscr{S}_{p}$). For every element $h \in \{1, \ldots, p+q\}$, we denote by $\mathscr{W}^{h}$ the stabilizer of $h$, i.e.
\begin{equation*}
\mathscr{W}^{h} = \left\{\sigma \in \mathscr{W}, \sigma(h) = h\right\}.
\end{equation*}
We define similarly $\mathscr{W}(\mathfrak{k})^{h}$.
\end{nota}

\begin{prop}

We get, up to a constant, the following result:
\begin{equation*}
D(X)\Theta_{\Pi'}(\exp(X)) = \begin{cases}  \sum\limits_{\mu \in A^{p+1}}\sgn(\mu) \sum\limits_{\omega \in \mathscr{S}_{p} \times \mathscr{S}_{q}} \sgn(\omega) e^{i\omega(\mu \lambda_{1} + \xi)(x)} & \text{ if } k > p-1 \\
 \sum\limits_{\mu \in A^{1}}\sgn(\mu) \sum\limits_{\omega \in \mathscr{S}_{p} \times \mathscr{S}_{q}} \sgn(w) e^{i\omega(\mu \lambda_{2} + \xi)(x)} & \text{ otherwise }
\end{cases}
\end{equation*}
where 
\begin{itemize}
\item $\lambda_{1}  = \sum\limits_{i=1}^{p} (i-1)e_{i} + (p-(k+1))e_{p+1} + \sum\limits_{i = p+2}^{p+q} (i-2)e_{i}$, 
\item $\lambda_{2} = (p-(k+1))e_{1} + \sum\limits_{a=2}^{p+q} (a-2)e_{a}$, 
\item $A^{1}$ (resp. $\A^{p+1}$) is a system of representatives of $\mathscr{W}^{1}/\mathscr{W}(\mathfrak{k})^{1}$ (resp. $\mathscr{W}^{p+1}/\mathscr{W}(\mathfrak{k})^{p+1}$)
\item $\xi = \sum\limits_{k=1}^{p+q} \frac{p+q-2}{2} e_{k}$.
\end{itemize}

\label{CorollaryFourier}

\end{prop}

\begin{proof}

We assume first that $k > p-1$. According to Equation \eqref{WeylDenominator}, we have:
\begin{equation*}
\prod\limits_{\alpha \in \Phi^{+}(\mathfrak{g}_{\mathbb{C}}, \mathfrak{t}_{\mathbb{C}})}(e^{\frac{\alpha(x)}{2}} - e^{-\frac{\alpha(x)}{2}}) = \prod\limits_{i=1}^m t^{-\frac{p+q-1}{2}}_{i} \prod\limits_{1 \leq i < j \leq p+q} (t_{i} - t_{j}).
\end{equation*}
For all $h \in \{1, \ldots, p+q\}$, we get:
\begin{equation*}
\cfrac{\prod\limits_{1 \leq i < j \leq p+q} (t_{i} - t_{j})}{\prod\limits_{k \neq h} (t_{h} - t_{k})} = (-1)^{h-1} \prod\limits_{\underset{i, j \neq h}{1 \leq i < j \leq p+q}} (t_{i} - t_{j})
\end{equation*}
and then
\begin{equation*}
\prod\limits_{\alpha \in \Phi^{+}(\mathfrak{g}_{\mathbb{C}}, \mathfrak{t}_{\mathbb{C}})}(e^{\frac{\alpha(x)}{2}} - e^{-\frac{\alpha(x)}{2}}) \Theta_{\Pi'_{k}}(\widetilde{\exp(x)}) = \prod\limits_{i=1}^{p+q}t^{-\frac{p+q-2}{2}}_{i} \sum\limits_{h=p+1}^{p+q}  (-1)^{h+1} t^{p-(k+1)}_{h} \prod\limits_{\underset{i \neq h, j \neq h}{1 \leqslant i < j \leqslant p+q}} (t_{i} - t_{j}).
\end{equation*}
For all $h \in \{1, \ldots, p+q\}$, we denote by $\tilde{t_{k}}, 1 \leq k \leq p+q-1$ the following elements
\begin{equation*}
\tilde{t}_{k} = \begin{cases}  t_{k} & \text{ if }  k < h \\ t_{k+1} & \text{ otherwise } \end{cases}.
\end{equation*}
Then, up to a $\pm 1$, we get:
\begin{eqnarray*}
\prod\limits_{\underset{i \neq h, j \neq h}{1 \leqslant i < j \leqslant p+q}} (t_{i} - t_{j}) & = & \prod\limits_{1 \leqslant i < j \leqslant p+q-1} (\tilde{t}_{i} - \tilde{t}_{j}) = \sum\limits_{\sigma \in \mathscr{S}_{p+q-1}} \sgn(\sigma) \prod\limits_{a=1}^{p+q-1} \tilde{t}^{a-1}_{\sigma(a)} \\ 
                                               & = & \sum\limits_{\sigma \in \mathscr{S}^{h}_{p+q}} \sgn(\sigma) \prod\limits_{a=1}^{h-1} t^{a-1}_{\sigma(a)} \prod\limits_{a=h+1}^{p+q} t^{a-2}_{\sigma(a)}.
\end{eqnarray*}
Finally, we prove that:
\begin{equation*}
\sum\limits_{h=p+1}^{p+q} \sum\limits_{\sigma \in \mathscr{S}^h_{p+q}} \sgn(\sigma) (-1)^{h+1} t^{p-(k+1)}_{h} \prod\limits_{a=1}^{h-1} t_{\sigma(a)}^{a-1} \prod\limits_{a = h + 1}^{p+q} t_{\sigma(a)}^{a-2} = (-1)^{p} \sum\limits_{\mu \in A^{p+1}}\sgn(\mu) \sum\limits_{\omega \in \mathscr{S}_{p} \times \mathscr{S}_{q}} \sgn(\omega) e^{i\omega(\mu \lambda)(x)}
\end{equation*}
where $\lambda_{1}  = \sum\limits_{i=1}^{p} (i-1)e_{i} + (p-(k+1))e_{p+1} + \sum\limits_{i = p+2}^{p+q} (i-2)e_{i}$. The proof is similar if $k \leq p-1$.

\end{proof}

\begin{rappels}

We recall briefly some well-known facts from \cite{ROS} (see also \cite{VER1}) concerning the Fourier transform of co-adjoint orbits. To simplify the notations, we assume that $\G$ is a semi-simple connected Lie group such that $\rk(\K) = \rk(\G)$, where $\K$ is a maximal compact subgroup of $\G$. We denote by $\Ad^{*}$ the natural co-ajoint action of $\G$ on $\mathfrak{g}^{*}$. For every $\lambda \in \mathfrak{g}^{*}$, we denote by $\G_{\lambda}$ the $\G$-orbit associated to $\lambda$. On the space $\G_{\lambda}$, we have a natural measure $d\beta_{\lambda}$, usually called the Liouville measure on $\G_{\lambda}$ (see \cite[Section~7.5]{VER1}). 

\noindent The Fourier transform of $\G_{\lambda}$, denoted by $\F_{\G_{\lambda}}$, is the generalized function on $\mathfrak{g}$ defined by:
\begin{equation*}
\F_{\G_{\lambda}}(X) = \displaystyle\int_{\G_{\lambda}} e^{if(X)} d\beta_{\lambda}(f) \qquad (X \in \mathfrak{g}).
\end{equation*}

\noindent As proved in \cite[Page~217]{ROS}, if $\lambda \in {\mathfrak{t}^{*}}^{\reg}$, we have, up to a constant, the following quality:
\begin{equation*}
\left(\prod\limits_{\alpha \in \Phi^{+}(\mathfrak{g}_{\mathbb{C}}, \mathfrak{t}_{\mathbb{C}})} \alpha(X)\right) \F_{\G_{\lambda}}(X) = \sum\limits_{\omega \in \mathscr{W}(\mathfrak{k})} \varepsilon(\omega) e^{i \lambda(\omega(X))} \qquad (X \in \mathfrak{t}^{\reg}).
\end{equation*}

\noindent If the weight $\lambda$ is not regular, see \cite[Theorem~7.24]{VER}. To simplify the notations, we denote by $\pi(X)$ the quantity $\prod\limits_{\alpha \in \Phi^{+}(\mathfrak{g}_{\mathbb{C}}, \mathfrak{t}_{\mathbb{C}})} \alpha(X)$.

\end{rappels}

\begin{coro}

Using the notations of Proposition \ref{CorollaryFourier}, we get, up to a constant: 
\begin{equation*}
\pi(X)^{-1} D(X)\Theta_{\Pi'}(\exp(X)) = \begin{cases}  \sum\limits_{\mu \in A^{p+1}}\sgn(\mu) \F_{\G_{\mu(\lambda_{1} + \xi)}}(X) & \text{ if } p < k+1 \\
\sum\limits_{\mu \in A^{1}}\sgn(\mu) \F_{\G_{\mu(\lambda_{2} + \xi)}}(X) & \text{ if } k < -q+1
\end{cases}
\end{equation*}

\end{coro}

\begin{rema}
\begin{enumerate}
\item For the dual pair $(\G = \U(1, \mathbb{C}), \G' = \U(1, \mathbb{C}))$, using the Equation \eqref{U(1,1,C)}, we get:
\begin{equation*}
\Theta_{\Pi'_{k}}(\widetilde{t}') = t'^{-k-\frac{1}{2}} = \Theta_{\Pi_{k}}(t'^{-1}).
\end{equation*}
In particular, to be precise, with our method, we don't get $\Theta_{\Pi'}(\widetilde{t}')$ but $\Theta_{\Pi'}(\widetilde{t}'^{-1})$ (or $\Theta_{\Pi'}(\widetilde{t}')$ because the representation $\Pi'$ is unitary): in the embedding of $(\G, \G')$ in $\Sp(W)$, $g' \in \G'$ acts on $w \in W$ as $g'.w = wg'^{-1}$.  
\item The function $\mathfrak{t} \to X \to \pi(X)^{-1}D(X) \in \mathbb{C}$ is well-known in the literature, usually denoted by $p(x)$ (see \cite{KIR} or \cite{ROS}). More particularly, $p(X)$ can be defined as:
\begin{equation*}
p(X) = \det^{\frac{1}{2}} \left(\cfrac{\sinh(\ad(X/2))}{\ad(X/2)}\right).
\end{equation*}
\end{enumerate}
\end{rema}

\section{The case $\G' = \U(1, 1, \mathbb{C})$}

\label{SectionU(1,1)}

\noindent As recalled in Equation \eqref{UnitaireU(n)22}, the unitary group $\U(p, q, \mathbb{C})$ is defined by
\begin{equation*}
\U(p, q, \mathbb{C}) = \left\{A \in \GL(p+q, \mathbb{C}), A^{*}\Id_{p, q}A = \Id_{p, q}\right\}.
\end{equation*}

\noindent We fix the convention that $p \leq q$. In this case, there is, up to conjugaison, $q+1$ Cartan subgroups in $\U(p, q, \mathbb{C})$ (see \cite{HIRAI}). More precisely, the Cartan subgroups $\H_{k}, 0 \leq k \leq q$ are of the form $\H_{k} = \H^{-}_{k}\H^{+}_{k}$, where $\H^{-}_{k}$ and $\H^{+}_{k}$ are the subgroups of $\U(p, q, \mathbb{C})$ are defined by
\begin{equation*}
\H^{+}_{k} = \left\{\H^{+}_{k}(t_{1}, \ldots, t_{k}), t_{i} \in \mathbb{R}, 1 \leq i \leq k \right\}
\end{equation*}
where $\H^{+}_{k}(t_{1}, \ldots, t_{k})$ is given by
\begin{equation*}
\H^{+}_{k}(t_{1}, \ldots, t_{k}) = \begin{pmatrix} \Id_{p-k} & & & & & & & & & 0 \\ & \ch(t_{k}) & & & & & & & \sh(t_{k}) & \\ & & \ch(t_{k-1}) & & & & & \sh(t_{k-1}) & & \\ & & & \ddots & & & \reflectbox{$\ddots$} & & & \\ & & & & \ch(t_{1}) & \sh(t_{1}) & & & & \\ & & & & \sh(t_{1}) & \ch(t_{1}) & & & & \\ & & & \reflectbox{$\ddots$} & & & \ddots & & & \\ & & \sh(t_{k-1}) & & & & & \ch(t_{k-1}) & & \\ & \sh(t_{k}) & & & & & & & \ch(t_{k}) & \\ 0 & & & & & & & & & \Id_{q-k} \end{pmatrix}
\end{equation*}
and 
\begin{equation*}
\H^{-}_{k} = \{\diag(e^{i\phi_{1}}, \ldots, e^{i\phi_{p-k}}, e^{i\theta_{k}}, \ldots, e^{i\theta_{1}}, e^{i\theta_{1}}, \ldots, e^{i\theta_{k}}, e^{i\tau_{q-k}}, \ldots, e^{i\tau_{1}}), \phi_{i}, \theta_{i}, \tau_{i} \in \mathbb{R}\}
\end{equation*}

\noindent We now work with the pair $(\G(V, \b_{V}), \G(V', \b_{V'})) = (\U(1, \mathbb{C}), \U(1, 1, \mathbb{C})) \subseteq \Sp(W)$, where $W = (\mathbb{C} \otimes \mathbb{C}^{2})_{\mathbb{R}}$ and $\langle\cdot, \cdot\rangle = \Im(\b_{V} \otimes \b_{V'})$. We denote by $\H_{1}$ and $\H_{2}$ the two Cartan subgroups of $\G'$ (up to conjugation), with $\H_{1}$ compact.

\noindent Let $\mathscr{B}' = \{v_{1}, v_{2}\}$ be a basis of $V'$ such that:
\begin{equation*}
\F = \Mat(\mathscr{B}', \b_{V'}) = \begin{pmatrix} 1 & 0 \\ 0 & -1 \end{pmatrix}.
\end{equation*}
Then, 
\begin{equation*}
\U(1, 1, \mathbb{C}) = \{g \in \GL(2, \mathbb{C}), g^{*}\F g = \F\},
\end{equation*}  
and
\begin{equation*}
\mathfrak{u}(1, 1, \mathbb{C}) = \{A \in \M(2, \mathbb{C}), A\F + \F A^{*} = 0\} = \mathbb{R}\begin{pmatrix} i & 0 \\ 0 & i \end{pmatrix} \oplus \mathbb{R}\begin{pmatrix} i & 0 \\ 0 & -i \end{pmatrix} \oplus \mathbb{R} \begin{pmatrix} 0 & 1 \\ 1 & 0 \end{pmatrix} \oplus \mathbb{R}\begin{pmatrix} 0 & i \\ -i & 0 \end{pmatrix}.
\end{equation*}

\noindent In particular, we have $\mathfrak{u}(1, 1, \mathbb{C})_{\mathbb{C}} = \mathbb{C}\Id_{2} \oplus \mathfrak{sl}(2, \mathbb{C})$. We fix the $\mathfrak{sl}(2, \mathbb{C})$-triple $(h, e, f)$ defined by $h = E_{1, 1} - E_{2, 2}$, $e = E_{1, 2}$, $f = E_{2, 1}$ and let 
\begin{equation*}
C = \exp\left(i\frac{\pi}{4}(e+f)\right).
\end{equation*}
More particularly, we have:
\begin{equation*}
C = \exp \begin{pmatrix} 0 & \frac{i\pi}{4} \\ \frac{i\pi}{4} & 0 \end{pmatrix} = \begin{pmatrix} \ch(\frac{i\pi}{4}) & \sh(\frac{i\pi}{4}) \\ \sh(\frac{i\pi}{4}) & \ch(\frac{i\pi}{4}) \end{pmatrix} = \begin{pmatrix} \cos(\frac{\pi}{4}) & i \sin(\frac{\pi}{4}) \\ i \sin(\frac{\pi}{4}) & \cos(\frac{\pi}{4}) \end{pmatrix} = \frac{1}{\sqrt{2}} \begin{pmatrix} 1 & i \\ i & 1 \end{pmatrix},
\end{equation*}
so $C \notin \mathfrak{u}(1, 1, \mathbb{C})$.

\begin{lemme}

Let
\begin{equation*}
\mathfrak{h}_{1} = \mathbb{R} \begin{pmatrix} i & 0 \\ 0 & i \end{pmatrix} \oplus \mathbb{R} \begin{pmatrix} i & 0 \\ 0 & -i \end{pmatrix} \qquad \mathfrak{h}_{2} = \mathbb{R} \begin{pmatrix} i & 0 \\ 0 & i \end{pmatrix} \oplus \mathbb{R} \begin{pmatrix} 0 & i \\ -i & 0 \end{pmatrix}
\end{equation*}
Then, $\H_{1} = \exp(\mathfrak{h}_{1})$ and $\H_{2} = \exp(\mathfrak{h}_{2})$ are the two non-conjugate Cartan subgroups of $\U(1, 1, \mathbb{C})$ and $\H_{1}$ is compact. Moreover, we have
\begin{equation*}
\begin{pmatrix} 0 & 1 \\ -1 & 0 \end{pmatrix} = C \begin{pmatrix} i & 0 \\ 0 & -i \end{pmatrix} C^{-1}
\end{equation*}

\end{lemme}

\begin{rema}

More particularly, the subgroups $\H_{1}$ and $\H_{2}$ are given by
\begin{equation*}
\H_{1} = \left\{\begin{pmatrix} e^{i\theta_{1}} & 0 \\ 0 & e^{i\theta_{2}} \end{pmatrix}, \theta_{1}, \theta_{2} \in \mathbb{R}\right\}
\end{equation*}
and 
\begin{equation*}
\H_{2} = \left\{\begin{pmatrix} e^{i\theta_{1}} \ch(X) & i \sh(X) \\ -i\sh(X) & e^{i\theta_{1}} \ch(X) \end{pmatrix}, \theta_{1}, X \in \mathbb{R} \right\} = \left\{\begin{pmatrix} e^{i\theta_{1}} & 0 \\ 0 & e^{i\theta_{1}} \end{pmatrix}, \theta_{1} \in \mathbb{R} \right\} . \left\{\begin{pmatrix} \ch(X) & i\sh(X) \\ -i\sh(X) & \ch(X) \end{pmatrix}, X \in \mathbb{R} \right\} = \T_{2} \A_{2}.
\end{equation*}
The set $\A_{2}$ is the split part of $\H_{2}$ (see \cite[Section~2.3.6]{WAL}).

\end{rema}

\begin{prop}

For all element $\widetilde{g} \in \widetilde{\G}'^{++}$, we get:
\begin{equation*}
\Theta(\widetilde{g}) = \det(g)^{\frac{1}{2}} \det(g-1)^{-1}.
\end{equation*}

\end{prop}

\begin{proof}

Let $g \in \U(1, 1, \mathbb{C})$ $\left(\leftrightarrow \begin{pmatrix} g & 0 \\ 0 & g^{*} \end{pmatrix} \in \Sp(4, \mathbb{R})\right)$. We get:
\begin{eqnarray*}
\det_{W_{\mathbb{C}}}(i(g-1)) & = & \det_{W_{\mathbb{C}}}(g-1) = \det_{\mathbb{C}^{2}}(g-1)  \det_{\mathbb{C}^{2}}(g^{*}-1) \\
                                               & = & \det(g-1) \det(g^{-1} - 1) = \det(g-1)\det(g^{-1}) \det(1-g) \\
                                               & = & \det(g)^{-1} \det(g-1)^{2}.
\end{eqnarray*}                                               

\end{proof}

\noindent We can now determine $\Theta_{\Pi'_{k}}$ on $\H_{2}$ (more particularly on $\A_{2}$). According to Equation \eqref{FormuleIntegrale1908}, we get for $X > 0$:
\begin{eqnarray*}
\Theta_{\Pi'_{k}}\left(\begin{pmatrix} \ch(X) & i\sh(X) \\ -i\sh(X) & \ch(X) \end{pmatrix}\right) & = & \Theta_{\Pi'_{k}}\left(\exp\begin{pmatrix} 0 & iX \\ -iX & 0\end{pmatrix}\right) = \displaystyle\int_{\S^{1}} z^{-k} \Theta\left(z\exp\begin{pmatrix} 0 & iX \\ -iX & 0\end{pmatrix}\right) dz \\
                               & = & \displaystyle\int_{\S^{1}} z^{-k} \Theta\left(z \exp C\begin{pmatrix} -X & 0 \\ 0 & X \end{pmatrix}C^{-1}\right) dz = \displaystyle\int_{\S^{1}} z^{-k} \Theta\left( \begin{pmatrix} ze^{-X} & 0 \\ 0 & ze^{X} \end{pmatrix}\right) dz \\
                               & = & \displaystyle\int_{\S^{1}} z^{-k+1} \det \begin{pmatrix} ze^{-X} - 1 & 0 \\ 0 & ze^{X} - 1\end{pmatrix}^{-1} dz = \displaystyle\int_{\S^{1}} \cfrac{z^{-k+1}}{(ze^{X}-1)(ze^{-X} - 1)} dz \\
                               & = & \displaystyle\int_{\S^{1}} \cfrac{z^{-k+1}}{(z-e^{-X})(z - e^{X})} dz = \begin{cases} \cfrac{e^{X(k-1)}}{e^{X} - e^{-X}} & \text{ if } k \geq 1 \\ \cfrac{e^{-X(k-1)}}{e^{-X} - e^{X}} & \text{ otherwise } \end{cases}
\end{eqnarray*}    
                           
%More precisely, we get:
%\begin{equation*}
%\Theta_{\Pi'_{k}}\left(\exp\begin{pmatrix} 0 & iX \\ -iX & 0\end{pmatrix}\right) = \cfrac{e^{-|X|(k-1)}}{e^{|X|} - e^{-|X|}} = \cfrac{2e^{-|X|(k-1)}}{\sh(|X|)} \qquad (X \neq 0).
%\end{equation*}

\noindent More generally, for all $\theta \in \mathbb{R}$ and $X \in \mathbb{R}^{+}$, we get:
\begin{eqnarray*}
& & \Theta_{\Pi'_{k}}\left(\begin{pmatrix} e^{i\theta}\ch(X) & ie^{i\theta}\sh(X) \\ -ie^{i\theta}\sh(X) & e^{i\theta}\ch(X) \end{pmatrix}\right) = \Theta_{\Pi'_{k}}\left(\exp\left( \begin{pmatrix} i\theta & 0 \\ 0 & i\theta \end{pmatrix} +\begin{pmatrix} 0 & iX \\ -iX & 0\end{pmatrix}\right)\right) \\
& = &  \displaystyle\int_{\S^{1}} z^{-k} \Theta\left(z\exp\left( \begin{pmatrix} i\theta & 0 \\ 0 & i\theta \end{pmatrix} +\begin{pmatrix} 0 & iX \\ -iX & 0\end{pmatrix}\right)\right) dz = \displaystyle\int_{\S^{1}} z^{-k} \Theta\left(z \exp \left(C \left( \begin{pmatrix} i\theta & 0 \\ 0 & i\theta \end{pmatrix} +\begin{pmatrix} -X & 0 \\ 0 & X \end{pmatrix}\right)C^{-1}\right)\right) dz \\ 
& = & \displaystyle\int_{\S^{1}} z^{-k} \Theta\left( \begin{pmatrix} ze^{-X}e^{i\theta} & 0 \\ 0 & ze^{i\theta}e^{X} \end{pmatrix}\right) dz = \displaystyle\int_{\S^{1}} z^{-k+1}e^{i\theta} \det \begin{pmatrix} ze^{i\theta}e^{-X} - 1 & 0 \\ 0 & ze^{i\theta}e^{X} - 1\end{pmatrix}^{-1} dz \\
& = & \displaystyle\int_{\S^{1}} \cfrac{z^{-k+1}e^{i\theta}}{(ze^{i\theta}e^{X}-1)(ze^{i\theta}e^{-X} - 1)} dz = \displaystyle\int_{\S^{1}} \cfrac{e^{-i\theta}z^{-k+1}}{(z-e^{-i\theta}e^{-X})(z - e^{-i\theta}e^{X})} dz = \begin{cases} \cfrac{e^{i\theta(k-1)}e^{X(k-1)}}{e^{X} - e^{-X}} & \text{ if } k \geq 1 \\ \cfrac{e^{i\theta(k-1)}e^{-X(k-1)}}{e^{-X} - e^{X}} & \text{ otherwise } \end{cases}
\end{eqnarray*}
We got similar results for $X < 0$.

\section{A conjecture of T. Przebinda}

\label{ConjectureT}

\noindent In \cite{TOM3}, T. Przebinda investigated the correspondence of characters for a general dual pair. We recall here the Howe's duality theorem in this context. Let $(W, \langle \cdot, \cdot\rangle )$ be a symplectic vector space over $\mathbb{R}$, $\widetilde{\Sp(W)}$ the corresponding metaplectic group, $(\omega, \mathscr{H})$ the metaplectic representation of $\widetilde{\Sp(W)}$, $(\G, \G')$ be a dual pair in $\Sp(W)$ and $(\widetilde{\G}, \widetilde{\G}')$ the corresponding dual pair in $\widetilde{\Sp(W)}$.

\noindent We denote by $\mathscr{R}(\widetilde{\G}, \omega)$ the set of equivalence classes of irreducible admissible representations of $\widetilde{\G}$ which are infinitisemally equivalent to a quotient of $\omega^{\infty}$. In \cite{HOW1}, R. Howe proved that there exists a bijection between $\mathscr{R}(\widetilde{\G}, \omega)$ and $\mathscr{R}(\widetilde{\G}', \omega)$ whose graph is $\mathscr{R}(\widetilde{\G}. \widetilde{\G}', \omega)$. We denote by $\theta$ the following one-to-one map:
\begin{equation*}
\theta: \mathscr{R}(\widetilde{\G}, \omega) \ni \Pi \to \Pi' = \theta(\Pi) \in \mathscr{R}(\widetilde{\G}', \omega).
\end{equation*}

\noindent Let's $(\G, \G')$ be an irreducible reductive dual pair in $\Sp(W)$. Without loss of generality, we assume that $\rk(\G) \leq \rk(\G')$. We denote by $\{\H_{1}, \ldots, \H_{n}\}$ the set of conjugacy classes of Cartan subgroups of $\G$. As explained in \cite[Section~2.3.6]{WAL}, for every $1 \leq i \leq n$, there exists a decomposition of $\H_{i}$ of the form
\begin{equation*}
\H_{i} = \T_{i}\A_{i},
\end{equation*}
where $\T_{i}$ is compact and $\A_{i}$ is the split part of $\H_{i}$ (by convention, $\H_{1}$ is the compact Cartan subgroup, i.e. $\A_{1} = \{\Id\}$).

\noindent For every $1 \leq i \leq n$, we denote by $\A^{'}_{i} = C_{\Sp(W)}(\A_{i})$ and $\A^{''}_{i} = C_{\Sp(W)}(\A^{'}_{i})$ (in particular, $\A^{'}_{1} = \Sp(W)$ and $\A^{''}_{1} = Z(\Sp(W)) = \{\pm 1\}$). Then, $(\A^{'}_{i}, \A^{''}_{i})$ is a reductive dual pair in $\Sp(W)$ (not irreducible in general).

\noindent We define a measure $\overline{dw_{i}}$ on the quotient $\A^{''}_{i} \setminus W$ given by:
\begin{equation*}
\displaystyle\int_{W} \phi(w) dw = \displaystyle\int_{\A^{''} \setminus W} \displaystyle\int_{\A^{'}_{i}} \phi(a'w) d\mu_{\A^{'}_{i}}(a') \overline{dw}.
\end{equation*}

\noindent In \cite[Section~2]{TOM3}, T. Przebinda define the following distribution on $\A^{'}_{i}$:
\begin{equation*}
\Chc(\Psi) = \displaystyle\int_{\A^{''}_{i} \setminus W} \left(\displaystyle\int_{\A^{'}_{i}} \Psi(\widetilde{g}) T(\widetilde{g}) d\widetilde{g}\right)(w) \overline{dw} \qquad (\Psi \in \mathscr{C}^{\infty}_{c}(\A^{'}_{i})).
\end{equation*}
As mentioned in Section \ref{IntegralFormula}, the integral 
\begin{equation*}
\displaystyle\int_{\A^{'}_{i}} \Psi(\widetilde{g}) T(\widetilde{g}) d\widetilde{g} \in \S(W),
\end{equation*}
in particular, $\Chc(\Psi)$ is well defined. Moreover, for all $\widetilde{h} \in \H^{\reg}_{i}$, the intersection of the wave front set $\WF(\Chc)$ of the distribution $\Chc$ with the conormal bundle of the embedding
\begin{equation*}
\widetilde{\G'} \ni \widetilde{g}' \to \widetilde{h}\widetilde{g}' \in \widetilde{A'_{i}}
\end{equation*}
is empty. In particular, there is a unique restriction of the distribution $\Chc$ to $\widetilde{\G'}$. We denote by $\Chc_{\widetilde{h}}$ this restriction. In \cite[Conjecture~2.18]{TOM3}, T. Przebinda conjectured the following result:

\begin{conj}[T. Przebinda]

We denote by $\G'_{1}$ the connected component at identity of $\G'$. We assume that $(\Theta_{\Pi'})_{\widetilde{\G'} \setminus \widetilde{\G'_{1}}} = 0$. For every $\Psi \in \mathscr{C}^{\infty}_{c}(\widetilde{\G'_{1}})$, the character $\Theta_{\Pi'}$ of $\Pi'$ is given by:
\begin{equation}
\Theta_{\Pi'}(\Psi) = \K_{\Pi} \sum\limits_{i=1}^n \cfrac{1}{|\mathscr{W}(\H_{i})|}\displaystyle\int_{\widetilde{\H}^{\reg}_{i}} \overline{\Theta_{\Pi}(\widetilde{h})} |D(\widetilde{h})|^{2} \Chc_{\widetilde{h}}(\Psi) d\widetilde{h}.
\label{ConjectureTomasz1}
\end{equation}
where $\K_{\Pi}$ is a complex number depending of $\Pi$ (one can check \cite[Definition~2.17]{TOM3}).

\end{conj}

\noindent We now explain how can we get characters by double lifting starting with a compact dual pair. To simplify the notations, we will present that for the dual pair of unitary groups. Let  $(\G = \U(1, \mathbb{C}), \G' = \U(1, 1, \mathbb{C}))$ in $\Sp(W_{\mathbb{R}})$, where $W_{\mathbb{R}} = (\mathbb{C}^{1} \otimes_{\mathbb{C}} \mathbb{C}^{1, 1})_{\mathbb{R}}$ and $(\G_{1}, \G'_{1}) = (\U(1, 1, \mathbb{C}), \U(m, m+1))$ in $\Sp(W^{m}_{\mathbb{R}})$, where $W^{1}_{\mathbb{R}} = (\mathbb{C}^{1, 1} \otimes_{\mathbb{C}} \mathbb{C}^{m, 1+m})_{\mathbb{R}}$. We denote by $(\omega, \mathscr{H})$ the metaplectic representation of $\widetilde{\Sp(W_{\mathbb{R}})}$, by $(\omega_{m}, \mathscr{H}_{m})$ the metaplectic representation of $\widetilde{\Sp(W^{m}_{\mathbb{R}})}$ and by $\theta$ and $\theta_{m}$ the two bijections
\begin{equation*}
\theta: \mathscr{R}(\widetilde{\G}, \omega) \to \mathscr{R}(\widetilde{\G'}, \omega) \qquad \theta_{m}: \mathscr{R}(\widetilde{\G_{1}}, \omega_{m}) \to \mathscr{R}(\widetilde{\G'_{1}}, \omega_{m}).
\end{equation*}

\begin{nota}

For two positive integers $r, s$, we denote by $\det^{\frac{r-s}{2}}-\cov$ the double cover of $\U(r, s, \mathbb{C})$ given by:
\begin{equation*}
\left\{(g, \xi) \in \U(r, s, \mathbb{C}) \times \mathbb{C}^{*}, \xi^{2} = \det(g)^{r-s}\right\}.
\end{equation*}

\end{nota}

\noindent According to \cite[Section~1.2]{PAUL}, we have:
\begin{equation*}
\widetilde{\G} \approx \det^{\frac{1}{2}}-\cov, \qquad \widetilde{\G'} \approx \det^{0}-\cov, \qquad \widetilde{\G_{1}} \approx \det^{0}-\cov \qquad \widetilde{\G'_{1}} \approx \det^{\frac{1}{2}}-\cov.
\end{equation*}
where $\det^{0}-\cov$ is the trivial cover.

\noindent Let $\Pi \in \mathscr{R}(\widetilde{\G}, \omega)$ such that $\theta(\Pi) \neq \{0\}$. According to a result of J-S Li \cite{LI}, $\theta(\Pi) \in \mathscr{R}(\widetilde{\G_{1}}, \omega_{m})$ and using the conservation of Kudla \cite{KUD},  $\theta_{m}(\theta(\Pi)) \neq \{0\}$.

\noindent We now assume that the dual pair $(\G_{1}, \G'_{1})$ is in the stable range (with $\rk(\G_{1}) \leq \rk(\G'_{1})$).

\noindent As explained previously, the weights of the representations $\Pi$ such that $\theta(\Pi) \neq \{0\}$ are well-known (see \cite{VER}). The distribution character $\Theta_{\theta_{m}(\theta(\Pi))}$ is given, for all $\Psi \in \mathscr{C}^{\infty}_{c}(\widetilde{\G'_{1}})$, by the following formula:
\begin{equation*}
\Theta_{\theta_{m}(\theta(\Pi))}(\Psi) = \K_{\Pi} \sum\limits_{i=1}^{2} \cfrac{1}{|\mathscr{W}(\H_{i})|}\displaystyle\int_{\widetilde{\H}^{\reg}_{i}} \overline{\Theta_{\theta(\Pi)}(\widetilde{h_{i}})} |D(\widetilde{h_{i}})|^{2} \Chc_{\widetilde{h_{i}}}(\Psi) d\widetilde{h_{i}}.
\end{equation*}
where $\{\H_{1}, \H_{2}\}$ the set of Cartan subgroups of $\G'$ (up to equivalence). The value of the character $\Theta_{\theta(\Pi)}$ on $\H_{1}$ can be obtained using the formula given in Proposition \ref{CharacterU(1,C)} (it can also be obtained using Enright's formula (see \cite[Corollary~2.3]{ENR}). The representation $\theta(\Pi)$ is an irreducible unitary highest weight module. According to H. Hecht's thesis \cite{HEC}, the formula on the other Cartan subgroups is given by the same formula on the other Cartan subgroups.
$\H_{2}$ can be obtained similarly.
 
\noindent Clearly, the same method can be applied in "higher dimensions".

\appendix

\section{The oscillator semigroup for $\U(1, 1, \mathbb{C})$} 

\label{AppendixA}

\noindent In this appendix, we would like to prove the equality given in Equation \ref{Semi} of the Proposition \ref{Proposition19081990}.

\noindent We recall here some well-known facts concerning complexifications. Let's $V$ be a complex dimension $n$ endowed with an antilinear involution $c$ (i.e. $c^{2} = 1$ and $c(\lambda v) = \bar{\lambda} c(v)$, $\lambda \in \mathbb{C}, v \in V$). Then, we have the decomposition:
\begin{equation}
V = \left\{v \in V, c(v) = v\right\} \oplus \left\{v \in V, c(v) = -v\right\} = \Re(V, c) \oplus \Im(V, c).
\label{RealVectorSpace}
\end{equation}
Both $\Re(V, c)$ and $\Im(V, c)$ are vector spaces over $\mathbb{R}$ of dimension $n$. Moreover, 
\begin{equation*}
\Re(V, c) \ni v \to iv \in \Im(V, c)
\end{equation*}
is well-defined and an isomorphism of $\mathbb{R}$ vector spaces. We denote by $V_{\mathbb{R}}$ the vector space given in Equation \eqref{RealVectorSpace}.

\begin{exem}

Let $V = \mathbb{C}^{n}$ and $c: V \to V$ given by $c(v) = \bar{v}$ the natural conjugation on $\mathbb{C}^{n}$. Then, 
\begin{equation*}
\Re(V) = \left\{x \in \mathbb{C}^{n}, x = \bar{x}\right\} = \mathbb{R}^{n} \qquad \Im(V) = \left\{x \in \mathbb{C}^{n}, x = -\bar{x}\right\} = i\mathbb{R}^{n}.
\end{equation*}

\end{exem}

\noindent On the vector space $V \oplus V$, we define the map $\tilde{c}$ given by
\begin{equation*}
\tilde{c}: V \oplus V \ni (u, v) \to (c(v), c(u)) \in V \oplus V
\end{equation*}
is an antilinear involution. As in Equation \eqref{RealVectorSpace}, we define the spaces $\Re(V \oplus V, \tilde{c})$ and $\Im(V \oplus V, \tilde{c})$. In particular, we have:
\begin{equation*}
\Re(V \oplus V, \tilde{c}) = \left\{(v, c(v)), v \in V\right\}.
\end{equation*}
Moreover, tha map:
\begin{equation}
V_{\mathbb{R}} \ni v \to (v, c(v)) \in \Re(V \oplus V, \tilde{c}) \subseteq V \oplus V
\label{VRiso}
\end{equation} 
is an isomorphism of real vector spaces. In particular, the complexification of $V_{\mathbb{R}}$, denoted by $(V_{\mathbb{R}})_{\mathbb{C}}$, is $V \oplus V$.

\noindent For all $(x, y) \in V \oplus V$, it's "conjugate" $\tilde{c}(x, y)$ is equal to $(c(y), c(x))$. Moreover, there exists $a, b \in V$ such that
\begin{equation*}
(x, y) = (a, c(a)) + (b, -c(b)) \in \Re(V \oplus V, \tilde{c}) \oplus \Im(V \oplus V, \tilde{c}).
\end{equation*}
More particularly, we have $a = \frac{u +c(v)}{2}$ and $b = \frac{u -c(v)}{2}$.

\noindent Let's now assume that the space $V$ is endowed with an hermitian form $\b_{V}$ of signature $(p, q)$ (i.e. there exists a basis $\mathscr{B}_{V}$ of $V$ such that $\F = \Mat(\b_{V}, \mathscr{B}_{V}) = \Id_{p, q}$). On the space $W = V_{\mathbb{R}}$, we have a natural symplectic form $\b$ given by $\b = \Im(\b_{V})$. We denote by $W_{\mathbb{C}}$ the complexification of $W$ and by $\b_{\mathbb{C}}$ th corresponding symplectic form on $W_{\mathbb{C}}$. Then, using the identification given in Equation \eqref{VRiso}, we get for all $w = (w_{1}, w_{2}) \in W_{\mathbb{C}} = V \oplus V$, we get:
\begin{eqnarray*}
H(w, w) & = & \Im\left(\b_{V}\left(\frac{w_{1} + c(w_{2})}{2}, -\frac{i(w_{1}-c(w_{2}))}{2}\right)\right) \\
             & = & \frac{1}{4}\left(\Re(\b_{V}(w_{1}, w_{1})) - \Re(\b_{V}(w_{1}, c(w_{2}))) + \Re(\b_{V}(c(w_{2}), w_{1})) - \Re(\b_{V}(c(w_{2}), c(w_{2})))\right) \\ 
             & = & \frac{1}{4}\left(\Re(\b_{V}(w_{1}, w_{1})) - \Re(\overline{\b_{V}(w_{1}, c(w_{2}))}) + \Re(\b_{V}(c(w_{2}), w_{1})) - \Re(\b_{V}(c(w_{2}), c(w_{2})))\right) \\ 
             & = & \frac{1}{4}\left(\Re(\b_{V}(w_{1}, w_{1})) - \Re(\b_{V}(c(w_{2}), c(w_{2})))\right) \\ 
             & = & \frac{1}{4}\left(\Re(\b_{V}(w_{1}, w_{1})) - \Re(\b_{V}(w_{2}, w_{2}))\right)
\end{eqnarray*}             

For all $g \in \Sp(W_{\mathbb{C}})^{++}$, we get, according to Equation \eqref{SpWC++}, that:
\begin{eqnarray*}
g \in \G^{++} & \Leftrightarrow & H\left(g\begin{pmatrix} w_{1} \\ w_{2}\end{pmatrix}, g\begin{pmatrix} w_{1} \\ w_{2}\end{pmatrix}\right) < H\left(\begin{pmatrix} w_{1} \\ w_{2}\end{pmatrix}, \begin{pmatrix} w_{1} \\ w_{2}\end{pmatrix}\right) \qquad (\forall (w_{1}, w_{2}) \in V \oplus V) \\
                                            & \Leftrightarrow & \Re(\b_{V}(w_{1}, w_{1})) - \Re(\b_{V}(w_{2}, w_{2})) > \Re(\b_{V}(gw_{1}, gw_{1})) - \Re(\b_{V}(g^{*}w_{2}, g^{*}w_{2})) \\
                                           & \Leftrightarrow & \begin{cases} \Re(\b_{V}(w_{1}, w_{1})) - \Re(\b_{V}(gw_{1}, gw_{1})) & > 0 \\ \Re(\b_{V}(w_{2}, w_{2})) - \Re(\b_{V}(g^{*}w_{2}, g^{*}w_{2})) & < 0 \end{cases} \\
                                           & \Leftrightarrow & \begin{cases} \b_{V}(w, w) - \b_{V}(gw, gw) & > 0 \\ \b_{V}(w, w) - \b_{V}(g^{*}w, g^{*}w) & < 0 \end{cases} \qquad (\forall w \in \Re(V, c) \setminus \{0\}) \\
                                          & \Leftrightarrow & \F - g^{*}\F g > 0
\end{eqnarray*}  

\begin{exem}

We assume that $V = \mathbb{C}^{2}$ and that the signature of $\b_{V}$ is $(1, 1)$. The compact torus $\T$ is given by:
\begin{equation*}
\T = \left\{t = \begin{pmatrix} t_{1} & 0 \\ 0 & t_{2} \end{pmatrix}, t_{1}, t_{2} \in \U(1, \mathbb{C})\right\}.
\end{equation*}
and then, using the notations of Section \ref{SectionU(1,1)},
\begin{equation*}
\F - \bar{t}^{t}\F t < 0 \Leftrightarrow \Id_{1, 1} - \begin{pmatrix} \bar{t_{1}} & 0 \\ 0 & \bar{t_{2}} \end{pmatrix} \Id_{1, 1} \begin{pmatrix} t_{1} & 0 \\ 0 & t_{2} \end{pmatrix}  > 0 \Leftrightarrow  \begin{pmatrix} 1 - |t_{1}|^{2}  & 0 \\ 0 &  |t_{2}|^{2} -1  \end{pmatrix} > 0,
\end{equation*}
i.e. $t \in \T^{++}  \Leftrightarrow |t_{1}| < 1$ and $|t_{2}| > 1$.   

\end{exem}

\end{document}